\documentclass{amsart}
\usepackage{lipsum}
\makeatletter
\g@addto@macro{\endabstract}{\@setabstract}
\newcommand{\authorfootnotes}{\renewcommand\thefootnote{\@fnsymbol\c@footnote}}%
\makeatother

\usepackage{enumitem}
\usepackage[dvipsnames]{xcolor}
\usepackage{setspace}
\usepackage[english]{babel}
\usepackage{amsmath,amsfonts,amssymb,amsthm,epsfig,epstopdf,titling,url,array}
\usepackage[margin=1.0in]{geometry}
\usepackage{tikz-cd}
\usepackage{calc}
\usepackage{color,hyperref}
\usepackage[noabbrev]{cleveref}
\hypersetup{colorlinks}
\hypersetup{colorlinks={true},linkcolor={red},citecolor=green}
\usepackage{graphicx}
\usepackage{mathrsfs}

\graphicspath{{images/}}
\theoremstyle{plain}
\newtheorem{theorem}{Theorem}

\newtheorem{thm}{Theorem}[subsection]
\newtheorem{lem}[thm]{Lemma}
\newtheorem{prop}[thm]{Proposition}
\newtheorem{cor}[thm]{Corollary}

\newcommand\scalemath[2]{\scalebox{#1}{\mbox{\ensuremath{\displaystyle #2}}}}
\theoremstyle{plain}
\newtheorem{defn}[thm]{Definition}

\theoremstyle{remark}
\newtheorem{rem}[thm]{Remark}

\newcommand{\mycomment}[1]{}
\usepackage{tikz}

\makeatletter
\renewcommand{\tocsection}[3]{%
  \indentlabel{\@ifnotempty{#2}{\bfseries\ignorespaces#1 #2\quad}}\bfseries#3}
\renewcommand{\tocsubsection}[3]{%
  \indentlabel{\@ifnotempty{#2}{\ignorespaces#1 #2\quad}}#3}

\renewcommand{\tocsubsubsection}[3]{%
  \indentlabel{\@ifnotempty{#2}{\ignorespaces#1 #2\quad}}#3}

\newcommand\@dotsep{4.5}
\def\@tocline#1#2#3#4#5#6#7{\relax
  \ifnum #1>\c@tocdepth 
  \else
    \par \addpenalty\@secpenalty\addvspace{#2}%
    \begingroup \hyphenpenalty\@M
    \@ifempty{#4}{%
      \@tempdima\csname r@tocindent\number#1\endcsname\relax
    }{%
      \@tempdima#4\relax
    }%
    \parindent\z@ \leftskip#3\relax \advance\leftskip\@tempdima\relax
    \rightskip\@pnumwidth plus1em \parfillskip-\@pnumwidth
    #5\leavevmode\hskip-\@tempdima{#6}\nobreak
    \leaders\hbox{$\m@th\mkern \@dotsep mu\hbox{.}\mkern \@dotsep mu$}\hfill
    \nobreak
    \hbox to\@pnumwidth{\@tocpagenum{\ifnum#1=1\bfseries\fi#7}}\par
    \nobreak
    \endgroup
  \fi}
\AtBeginDocument{%
\expandafter\renewcommand\csname r@tocindent0\endcsname{0pt}
}
\def\l@subsection{\@tocline{2}{0pt}{2.5pc}{5pc}{}}
\def\l@subsubsection{\@tocline{2}{0pt}{4pc}{5pc}{}}
\makeatother

\usepackage[new]{old-arrows}
\usepackage{bm} 
\usepackage{multicol}
\usetikzlibrary{decorations.pathmorphing}
\DeclareMathOperator{\GL}{GL}

\DeclareMathOperator{\ch}{ch}

\usepackage{mathtools}

\DeclareMathOperator{\vol}{vol}

\newcommand\blfootnote[1]{%
  \begingroup
  \renewcommand\thefootnote{}\footnote{#1}%
  \addtocounter{footnote}{-1}%
  \endgroup
}

\date{} 

\begin{document}
\hypersetup{citecolor=blue}
\hypersetup{linkcolor=red}

\begin{center}
  \LARGE 
 ON INTEGRAL ASPECTS OF ASAI PERIODS AND \\
 EULER SYSTEMS FOR $\mathrm{Res}_{E/\mathbf{Q}}\GL_2$.
 
  \normalsize
  \bigskip
  Alexandros Groutides \par 
  
 Mathematics Institute, University of Warwick\par
\end{center}

\begin{abstract}
    Let $E/\mathbf{Q}$ be a totally real quadratic field. Using unramified harmonic analysis in Hecke modules, we study the $\ell$-adic integral behavior of the (unramified part of the) Asai period attached to a Hilbert modular form for $E$, when evaluated on arbitrary integral test data in the sense of \cite{loeffler2021zetaintegralsunramifiedrepresentationsgsp4}. Using the same representation-theoretic framework, we also prove the conjectured integral behavior of local factors appearing in tame norm relations, between any collection of integral motivic Asai-Flach classes in the recipe of \cite{Loeffler_2021}. Finally, specializing to one such specific integral collection, we obtain the most general version of the Asai-Flach Euler system tame norm relations, extending a result of \cite{grossi2020norm}. 
\end{abstract}
\blfootnote{We gratefully acknowledge support from the following research grant: ERC Grant No. 101001051—Shimura varieties and the Birch–Swinnerton-Dyer conjecture.}
\vspace{-2em}
\section{Introduction}
\subsection{Main results}

This paper is devoted to the study of Asai periods associated to Hilbert modular forms for real quadratic fields. These play an important role in the study of $L$-functions of Hilbert modular forms, and the geometry and arithmetic of Hilbert modular surfaces; in particular they are fundamental input for the Euler system constructed \textcolor{Black}{in} \cite{lei2018euler}. Our goal is to establish integrality properties of these periods, when the input data lies in suitable \textcolor{Black}{lattices}.

 Throughout this introduction, we let $E/\mathbf{Q}$ be a totally real quadratic field with discriminant $\Delta_E$. We consider the natural embedding of algebraic groups $G:=\GL_2 \hookrightarrow\mathscr{G}:=\mathrm{Res}_{E/\mathbf{Q}}G$ and we write $\vol_{G(\mathbf{Q}_p)}$ for the Haar measure on $G(\mathbf{Q}_p)$ giving $G(\mathbf{Z}_p)$ volume $1$. Given a Hilbert cuspidal eigenform $\mathbf{f}$ for $E$, there is an associated automorphic representation $\Pi_\mathbf{f}$ of $\mathscr{G}(\mathbf{A})$ with central character denoted by $\omega_{\Pi_\mathbf{f}}$. This representation is unramified outside a large enough finite set of primes $S$ and its $S$-finite part $\Pi_\mathbf{f}^S$ contains a unique normalized spherical vector $W_{\Pi_\mathbf{f}^S}^\mathrm{sph}$. The unramified Asai period associated to $\mathbf{f}$ is the unique normalized non-zero linear form
\begin{align}\label{eq: intro 1}\mathcal{Z}_\mathbf{f}\in\mathrm{Hom}_{G(\mathbf{A}^S)}\left(\mathcal{S}((\mathbf{A}^S)^2)\otimes \Pi_\mathbf{f}^S,\mathbf{1}\right)\end{align}
in this one-dimensional Hom-space, where $\mathcal{S}((\mathbf{A}^S)^2)$ is a space of Schwartz functions. This period can be realized using the local zeta integrals and $L$-functions, of \cite{jacquet1983rankin} and \cite{flicker1988twisted}, together with work of \cite{matringe2008distinguishedrepresentationsexceptionalpoles} for the inert case. 

The $\ell$-adic integral behavior of $p$-adic integrals of this form, and their associated periods, have only been studied in specific cases of nice explicit input data (apart from our earlier work in \cite{groutides2024rankinselbergintegralstructureseuler} on which we build on here). This means that one can explicitly work with these periods, unfold the integrals involved using Godement-Siegel sections and obtain meaningful explicit results at an $\ell$-adic integral level. For example, results related to this, have appeared in \cite[\S $3$]{loeffler2021euler} , \cite[\S $4.3$]{grossi2020norm} , \cite[\S $6$]{Loeffler_2021} and \cite[\S $6.4$]{loeffler2021zetaintegralsunramifiedrepresentationsgsp4} .
Upon closer inspection, and with some translation, all such instances of $\ell$-adic integral behavior, fall under the more general framework of \textit{integral data} of \cite[\S $6.4$]{loeffler2021zetaintegralsunramifiedrepresentationsgsp4}. The major difficulty arises when one passes to \textit{general} such integral data in the sense of \textit{op.cit}. In this case, it becomes futile to try to attack matters of integrality by directly computing or unfolding the $p$-adic integrals involved in the construction. In \cite{groutides2024rankinselbergintegralstructureseuler} we develop a strategy to bypass this issue, obtaining general results for the corresponding period associated to the convolution of two modular forms \cite[Theorem B]{groutides2024rankinselbergintegralstructureseuler}, phrased in a slightly different language. In our first main theorem below, we adapt this approach and extend it to the Asai case.

\begin{theorem}[\Cref{thm hilbert period}]\label{thm intro A}
  Let $\mathbf{f}$ be a normalized, paritious, Hilbert cuspidal eigenform for $E$, of level $\mathfrak{n}$ and weight $(\underline{k},\underline{t})\in\mathbf{Z}_{\geq 2}^2\times\mathbf{Z}_{\geq 0}^2$. Let $L_\mathbf{f}$ be the number field of $\mathbf{f}$. Fix a prime $\ell\nmid\mathrm{Nm}_{E/\mathbf{Q}}(\mathfrak{n})$, a place $v$ of $L_\mathbf{f}$ above $\ell$, and set $\mathbf{L}_\mathbf{f}:=(L_\mathbf{f})_v$. Let $S$ be a finite set of places containing $\{\infty,p|2\ell\Delta_E\mathrm{Nm}_{E/\mathbf{Q}}(\mathfrak{n})\}$. Then the following are true:
\begin{enumerate}
    \item For any $g\in\mathscr{G}(\mathbf{A}^S)$ and any decomposable Schwartz function $\Phi=\otimes_{p\notin S} \Phi_p\in \mathcal{S}((\mathbf{A}^S)^2)$ where each $\Phi_p$ is valued in $\vol_{G(\mathbf{Q}_p)}(\mathrm{Stab}_{G(\mathbf{Q}_p)}(\Phi_p)\cap g_p\mathscr{G}(\mathbf{Z}_p)g_p^{-1})^{-1}\cdot\mathcal{O}_{L_\mathbf{f}}$, the period $\mathcal{Z}_\mathbf{f}$ satisfies 
    $$\mathcal{Z}_\mathbf{f}(\Phi\otimes gW_{\Pi_\mathbf{f}^S}^\mathrm{sph})\in\mathcal{O}_{\mathbf{L}_\mathbf{f}}.$$
    \item Suppose, moreover, that $S_0$ denotes a finite set of primes disjoint from $S$, and for each $p\in S_0$, $\Phi_p$ is valued in $\vol_{G(\mathbf{Q}_p)}(\mathrm{Stab}_{G(\mathbf{Q}_p)}(\Phi_p)\cap g_p\mathscr{G}(\mathbf{Z}_p)[p]g_p^{-1})^{-1}\cdot\mathcal{O}_{L_\mathbf{f}}$. Then the period $\mathcal{Z}_\mathbf{f}$ also satisfies
    \begin{align}\label{eq: A}\mathcal{Z}_\mathbf{f}\left(\Phi\otimes gW_{\Pi_\mathbf{f}^S}^\mathrm{sph}\right)\cdot \left(\prod_{\substack{p\in S_0\\ \Phi_p(0,0)\neq 0}} L_p(\omega_{\Pi_\mathbf{f}},0)^{-1}\right)\in\prod_{p\in S_0}\left\langle p-1, L_p^\mathrm{As}(\Pi_{\mathbf{f}},0)^{-1}\right\rangle\subseteq \mathcal{O}_{\mathbf{L}_\mathbf{f}}.\end{align}
\end{enumerate}
\end{theorem}

\textcolor{Black}{If the set $S_0\cap\{p|\Phi_p(0,0)\neq 0\}$ is empty, or if the narrow class number of $E$ and the Nebentype of $\mathbf{f}$ satisfy certain conditions found in \Cref{thm hilbert period}, the containment in \eqref{eq: A} also holds without the bracketed product of Tate $L$-factors on the left. For the definition of the open compact determinant level subgroup $\mathscr{G}(\mathbf{Z}_p)[p]\subseteq\mathscr{G}(\mathbf{Z}_p)$ and more details regarding conventions, normalizations and notation used in \Cref{thm intro A}, we refer the reader to \Cref{sec periods attached to HMF}.} The condition in the first part of the theorem can also be stated as: For every $p\notin S$ the tuple $(\Phi_p,g_p)$ is $\mathcal{O}_{L_\mathbf{f}}$-integral of level $\mathscr{G}(\mathbf{Z}_p)$, in the sense of \cite[ \S $6.4$]{loeffler2021zetaintegralsunramifiedrepresentationsgsp4}. Similarly, the condition on the values of $\Phi_p$ for $p\in S_0$, in the second part, can be stated as: For every $p\in S_0$ the tuple $(\Phi_p,g_p)$ is $\mathcal{O}_{L_\mathbf{f}}$-integral of level $\mathscr{G}(\mathbf{Z}_p)[p]$, in the sense of \textit{loc.cit}. Of course for primes $p\in S_0$, there's no clash between these two conditions since the latter implies the former. The unramified tuples $(\ch(\mathbf{Z}_p^2),g_p\in\mathscr{G}(\mathbf{Z}_p))$, which occur for almost all $p\notin S$, are by construction $\mathcal{O}_{L_\mathbf{f}}$-integral of level $\mathscr{G}(\mathbf{Z}_p)$, but are not $\mathcal{O}_{L_\mathbf{f}}$-integral of level $\mathscr{G}(\mathbf{Z}_p)[p]$.

\textcolor{Black}{Our second main result concerns the integral behavior of the Asai-Flach tame norm relations. Once again by adapting and extending our earlier work in \cite{groutides2024rankinselbergintegralstructureseuler} we show that the motivic Asai-Flach Euler system tame norm relations obtained using the approach of \cite{Loeffler_2021}}, are optimal in the strongest possible integral sense; i.e., any construction of this type, with any choice of integral input data, does not only yield motivic classes which lie in the integral \'etale realization, but also local factors which are always integrally divisible by the Asai Euler factor $\mathcal{P}_{p,\mathrm{As}^*}^{'}(\mathrm{Frob}_p^{-1})$, modulo $\scalemath{0.9}{p-1}$, settling a conjecture of Loeffler in this setting. \textcolor{Black}{We expect that similar integral optimality results can be established for norm relations associated to groups that enjoy a similar setup. For example 
 the $(\GL_2\times_{\GL_1}\GL_2,\mathrm{GSp}_4)$-setup of \cite{loeffler2021euler}, the $(\GL_2\times_{\GL_1}\GL_2,\mathrm{GSp}_4\times_{\GL_1} \GL_2)$ setup of \cite{hsu2020eulersystemsmathrmgsp4times} and of course the $(H,\mathrm{GU}(2,1))$-setup of \cite{Loeffler_2021}.}
 
 To obtain Asai-Flach classes over cyclotomic fields, we actually need to pass from $\mathscr{G}$ to a slightly smaller group $\mathscr{G}^*:=\mathscr{G}\times_D\GL_1$ where $D:=\mathrm{Res}_{E/\mathbf{Q}}\GL_1$.  The integral structures with determinant level introduced in \cite[\S $3$]{Loeffler_2021}  can be briefly described locally for the group $\mathscr{G}^*$, by lattices $$\scalemath{0.95}{\mathcal{I}_0(\mathscr{G}^*(\mathbf{Q}_p)/\mathscr{G}^*(\mathbf{Z}_p)[p],\mathbf{Z}[1/p]) \subseteq \left(\mathcal{S}_0(\mathbf{Q}_p^2)\otimes C_c^\infty(\mathscr{G}^*(\mathbf{Q}_p)/\mathscr{G}^*(\mathbf{Z}_p)[p])\right)_{\GL_2(\mathbf{Q}_p)}}$$given by 
$$\scalemath{0.9}{\mathrm{span}_{\mathbf{Z}[1/p]}\left\{ \phi\otimes \ch(g \mathscr{G}^*(\mathbf{Z}_p)[p])\ |\ g\in \mathscr{G}^*(\mathbf{Q}_p),\ \phi\ \mathrm{valued}\ \mathrm{in}\ \vol_{G(\mathbf{Q}_p)}\left(\mathrm{Stab}_{G(\mathbf{Q}_p)}(\phi)\cap g \mathscr{G}^*(\mathbf{Z}_p)[p]g^{-1}\right)^{-1}\cdot \mathbf{Z}[1/p] \right\}.}$$

\textcolor{Black}{For the definition of these spaces, we refer the reader to \Cref{sec general notation} and the beginning of \Cref{sec setup}.} Motivic cohomology classes attached to data in these local lattices are the ones we will be interested in. Let $\underline{\delta}^*:=(\delta_p^*)_{p\notin S}$ a family of local integral elements with determinant level. Upon properly truncating this family at each square-free $n\in\mathbf{Z}_{\geq 1}$ coprime to $S$, we can apply the recipe of \cite[\S $9$]{Loeffler_2021}  to obtain a collection of \textit{integral} classes (see \Cref{Prop integral classes}) 
 $\{ _c\mathrm{AF}_{\mathrm{mot},n}(\underline{\delta}^*)\}_{(n,S)=1}$, where for each such integer $n$, the class $ _c\mathrm{AF}_{\mathrm{mot},n}(\underline{\delta}^*)$ depends on $\delta_p^*$ for $p|n$. \textcolor{Black}{For the precise definition of these classes, we refer the reader to \eqref{AF cycl classes} and \Cref{def classes}} As mentioned above, but with some more notation, our second main theorem studies the integral behavior of local factors coming from any such collection of integral classes $\{ _c\mathrm{AF}_{\mathrm{mot},n}(\underline{\delta}^*)\}_{(n,S)=1}$, attached to any family of local integral data $\underline{\delta}^*$. This can be formulated for non-trivial coefficient sheaves, and upon doing so, our proof goes through in the same way.

\begin{theorem}[\Cref{thm euler system}]\label{thm intro B}
     Let $S$ be a finite set of primes containing $\{2,p|\Delta_E\}$, and $c\in\mathbf{Z}_{>1}$ coprime to $6S$. Let $Y_{\mathscr{G}^*}$ be the infinite-level Shimura variety of $\mathscr{G}^*$, and $\mathscr{H}_S$ an open compact level subgroup at $S$. Let $\mathscr{Z}_{cS}$ be the set of square-free positive integers coprime to $cS$, and let $\underline{\delta}^*=(\delta_p^*)_{p\notin S}$ be any collection of integral data with determinant level, $\delta_p^*\in\mathcal{I}_0(\mathscr{G}^*(\mathbf{Q}_p)/\mathscr{G}^*(\mathbf{Z}_p)[p],\mathbf{Z}[1/p])$. Then the integral motivic Asai-Flach classes 
    $$_c\mathrm{AF}_{\mathrm{mot},n}(\underline{\delta}^*)\in H_{\mathrm{mot}}^3(Y_{\mathscr{G}^*}(\mathscr{H}_S)\times_\mathbf{Q} \mathbf{Q}(\mu_n),\mathbf{Q}(2)),\ \ \ n\in\mathscr{Z}_{cS}$$
    where $_c\mathrm{AF}_{\mathrm{mot},n}(\underline{\delta}^*)$ depends on $\delta_p^*$ for $p|n$,
        satisfy the following norm relations:
        \begin{enumerate}
            \item For $n,m\in\mathscr{Z}_{cS}$ with $\tfrac{m}{n}=p$ prime, we have 
            $$\mathrm{norm}^{\mathbf{Q}(\mu_m)}_{\mathbf{Q}(\mu_n)}\left( _c\mathrm{AF}_{\mathrm{mot},m}(\underline{\delta}^*)\right)=\mathcal{P}_{\mathrm{Tr}(\delta_p^*)}^\mathrm{cycl}\cdot\ _c\mathrm{AF}_{\mathrm{mot},n}(\underline{\delta}^*)$$
            with $$\mathcal{P}_{\mathrm{Tr}(\delta_p^*)}^\mathrm{cycl}\in \left\langle p-1,\mathcal{P}_{p,\mathrm{As}_*}^{'}(\mathrm{Frob}_p^{-1})\right\rangle\subseteq \mathcal{H}_{\mathscr{G}^*(\mathbf{Q}_p)}^\circ(\mathbf{Z}[1/p])[\mathrm{Gal}(\mathbf{Q}(\mu_n)/\mathbf{Q})].$$
            \item If we specialize the integral collection $\underline{\delta}^*$ to $\underline{\delta}_1^*$
            then for all $n,m\in\mathscr{Z}_{cS}$ with $\tfrac{m}{n}=p$ prime, we have $$\mathrm{norm}^{\mathbf{Q}(\mu_m)}_{\mathbf{Q}(\mu_n)}\left( _c\mathrm{AF}_{\mathrm{mot},m}(\underline{\delta}_1^*)\right)=\mathcal{P}_{p,\mathrm{As}_*}^{'}(\mathrm{Frob}_p^{-1})\cdot\ _c\mathrm{AF}_{\mathrm{mot},n}(\underline{\delta}_1^*).$$
        \end{enumerate}
        \textcolor{Black}{Here $\mathcal{H}_{\mathscr{G}^*(\mathbf{Q}_p)}^\circ(\mathbf{Z}[1/p])$ denotes the spherical Hecke algebra of $\mathscr{G}^*(\mathbf{Q}_p)$ with $\mathbf{Z}[1/p]$-coefficients.} The local factor $\mathcal{P}^\mathrm{cycl}_{\mathrm{Tr}(\delta_p^*)}$ is canonical, $\mathrm{Frob}_p$ denotes arithmetic Frobenius at $p$ as an element of $\mathrm{Gal}(\mathbf{Q}(\mu_n)/\mathbf{Q})$, and $\mathcal{P}_{p,\mathrm{As}_*}^{'}(X)\in\mathcal{H}_{\mathscr{G}^*(\mathbf{Q}_p)}^\circ(\mathbf{Z}[1/p])[X]$ acts as an Euler factor upon passing to Galois cohomology. 
\end{theorem} 

The first instance of such a collection of integral Asai-Flach classes, satisfying motivic norm-relations, was established in \cite[Theorem $3.5.3$]{lei2018euler}. The approach used to establish norm-relations in \textit{op.cit}, differs from the one presented here, however the local factors appearing in \textit{loc.cit} are still in line (after some translating) with the integral ideal in \Cref{thm intro B}$(1)$. Additionally, in \textit{op.cit}, the authors obtain their motivic norm-relations for all inert primes, and for split primes which satisfy the so called ``narrowly principle'' condition. Here we obtain integral motivic norm relations for all unramified primes.

We should also mention the relation with work of Grossi in \cite{grossi2020norm}. In \textit{op.cit} the author constructs a specific collection of classes in motivic cohomology, which coincides (after adjusting for integrality which is not treated directly in \textit{op.cit}) with the classes $ _c\mathrm{AF}_{\mathrm{mot},n}(\underline{\delta}_1^*)$ constructed here. The author then also obtains Euler system tame norm relations for all unramified primes, with the expected Euler factor, but only after passing to Galois cohomology. \Cref{thm intro B}$(2)$ can thus be regarded as an extension of this. Indeed (assuming parallel weight $((2,2),(0,0))$ to ease notation and exposition), upon passing to Galois cohomology of \textcolor{Black}{the dual Asai representation} $\rho_\mathbf{f}^*$, the action of $\mathcal{P}_{p,\mathrm{As}_*}^{'}(\mathrm{Frob}_p^{-1})$ is intertwined with the action of the Artin Euler factor $P_p(\rho_\mathbf{f}^\mathrm{As}(1),\mathrm{Frob}_p^{-1})$, recovering \cite[Theorem $7.3.2$]{grossi2020norm}. 

\subsection{Acknowledgments}
I would like to thank my PhD supervisor David Loeffler for his constant guidance and support. Additionally, I would also like to thank Nadir Matringe for a useful discussion. Finally, I thank the anonymous referee for their careful reading of the manuscript.

{
  \hypersetup{linkcolor=black}
  \tableofcontents
}

\section{Some general notation}\label{sec general notation}
Throughout, we write $G:=\GL_2$ and $\vol_{G(\mathbf{Q}_p)}$ for the $\mathbf{Q}$-valued Haar measure on $G(\mathbf{Q}_p)$, giving $G(\mathbf{Z}_p)$ volume $1$. Likewise, we always write $\mathcal{S}(\mathbf{Q}_p^2)$, respectively $\mathcal{S}_0(\mathbf{Q}_p^2)$, for the space of Schwartz functions on $\mathbf{Q}_p^2$, respectively Schwartz functions on $\mathbf{Q}_p^2$ vanishing at $(0,0)$. We write $\mathcal{S}_{(0)}$ to denote either $\mathcal{S}$ or $\mathcal{S}_0$.

We briefly introduce some general local notation, in order to avoid repetition when passing from one group to another.
Let $\Gamma$ be an algebraic group, regarded over $\mathbf{Q}_p$, admitting $G$ as a subgroup via a fixed choice of embedding. Let $k$ be a finite extension of $\mathbf{Q}_p$ (possibly $k=\mathbf{Q}_p)$, and $K\subseteq\Gamma(k)$ be an open compact subgroup. \textcolor{Black}{We assume that $\Gamma(k)$ is unimodular since this will be the case for all $\Gamma$ considered in this paper. We write $C_c^\infty(\Gamma(k)/K)$ for the space of (smooth) compactly supported complex-valued functions on $\Gamma(k)$ that are right $K$-invariant.} As in \cite{Loeffler_2021} and \cite{groutides2024rankinselbergintegralstructureseuler}, we once again consider the spaces 
$$\mathcal{I}_{(0)}(\Gamma(k)/K):=\left(\mathcal{S}_{(0)}(\mathbf{Q}_p^2)\otimes C_c^\infty(\Gamma(k)/K)\right)_{G(\mathbf{Q}_p)}.$$
The subscript $G(\mathbf{Q}_p)$ denotes coinvariants, where the $G(\mathbf{Q}_p)$ action on $\mathcal{S}_{(0)}(\mathbf{Q}_p^2)\otimes C_c^\infty(\Gamma(k)/K)$ is given by $\gamma\cdot(\phi\otimes \xi):=\phi((-)\gamma)\otimes \xi(\gamma^{-1}(-))$. \textcolor{Black}{We have} the space $C_c^\infty(K\backslash\Gamma(k)/K)$ of compactly supported, $K$-bi-invariant $\mathbf{C}$-valued functions on $\Gamma(k)$. \textcolor{Black}{Upon fixing a $\mathbf{Q}$-valued Haar measure $d\gamma$ on $\Gamma(k)$, this is an associative unital algebra under the usual convolution}
$$\textcolor{Black}{(f_1f_2)(x):=\int_{\Gamma(k)}f_1(\gamma)f_2(\gamma^{-1}x)\ d\gamma,\ f_1,f_2\in C_c^\infty(K\backslash\Gamma(k)/K),x\in\Gamma(k).}$$
\textcolor{Black}{There is a natural action of this algebra on $C_c^\infty(\Gamma(k)/K)$ given by}
$$\textcolor{Black}{(f\cdot\xi)(x):=\int_{\Gamma(k)}f(\gamma)\xi(x\gamma)\ d\gamma,\ f\in C_c^\infty(K\backslash\Gamma(k)/K),\xi\in C_c^\infty(\Gamma(k)/K),x\in\Gamma(k).}$$
\textcolor{Black}{We can now regard $\mathcal{I}_{(0)}(\Gamma(k)/K)$ as a well-defined $C_c^\infty(K\backslash\Gamma(k)/K)$-module by setting $f\cdot (\phi\otimes\xi):=\phi\otimes(f\cdot\xi)$ for $f\in C_c^\infty(K\backslash\Gamma(k)/K)$}. These modules will be used throughout our work, for various choices of $\Gamma$. Given an element $f\in C_c^\infty(K\backslash\Gamma(k)/K)$, we always write $f^{'}$ for the element of $C_c^\infty(K\backslash\Gamma(k)/K)$ given by $f((-)^{-1})$. \textcolor{Black}{Given a subset $V\subseteq \Gamma(k)$, we write $\ch(V)$ for the characteristic function of $V$.}
\begin{defn}\cite[\S $3.2$]{Loeffler_2021}
    Let $A\subseteq\mathbf{C}$ be a $\mathbf{Z}$-algebra and $U\subseteq K$ an open compact subgroup of $\Gamma(k)$. The lattice $\mathcal{I}_{(0)}(\Gamma(k)/U,A)$ of $A$-integral elements at level $U$, is given by all $A$-linear combinations $\phi\otimes\ch(gU)\in \mathcal{S}_{(0)}(\mathbf{Q}_p^2)\otimes C_c^\infty(\Gamma(k)/U)$, with $\phi$ valued in 
    $$\frac{1}{\vol_{G(\mathbf{Q}_p)}(\mathrm{Stab}_{G(\mathbf{Q}_p)}(\phi)\cap gUg^{-1})}\cdot A\subseteq\mathbf{C}.$$
\end{defn}
\noindent As in \cite[Proposition $3.2.3$]{Loeffler_2021} , we have a trace map at integral level:
$$\mathrm{Tr}:=\mathrm{Tr}^U_K:\mathcal{I}_{(0)}(\Gamma(k)/U,A)\longrightarrow\mathcal{I}_{(0)}(\Gamma(k)/K,A),\ \phi\otimes \xi\mapsto\sum_{\gamma\in K/U} \phi\otimes \xi((-)\gamma).$$
We will always omit the dependence of this trace map, on $U$ and $K$, since it will be clear from context.
\section{Local theory}\label{sec local results}

Throughout this section we always let $p$ be an odd prime, we write $F/\mathbf{Q}_p$ for the unique unramified quadratic extension of $\mathbf{Q}_p$ and we fix $\alpha\in\mathcal{O}_F^\times$ such that $\mathcal{O}_F=\mathbf{Z}_p[\alpha]$. We denote by $B$ the upper triangular Borel of $G$, by $N$ its unipotent radical, and by $Z$ the center of $G$. We set 
$$\mathcal{H}_{G(F)}^\circ:=C_c^\infty(G(\mathcal{O}_F)\backslash G(F)/G(\mathcal{O}_F)),\ \mathcal{H}_{G(\mathbf{Q}_p)}^\circ:=C_c^\infty(G(\mathbf{Z}_p)\backslash G(\mathbf{Q}_p)/ G(\mathbf{Z}_p))$$
to be the spherical Hecke algebras of $G(F)$ and $G(\mathbf{Q}_p)$ under the usual convolution \textcolor{Black}{given in the previous section}, with respect to the Haar measure on $G(F)$, respectively $G(\mathbf{Q}_p)$, which gives $G(\mathcal{O}_F)$, respectively $G(\mathbf{Z}_p)$, volume $1$. We also  introduce the determinant open compact level subgroup
 $$G(\mathcal{O}_F)[p]:=\{g\in G(\mathcal{O}_F)\ |\ \det(g)\in 1+p\mathcal{O}_F\}\subseteq G(\mathcal{O}_F).$$
 The subgroup $G(\mathcal{O}_F)[p]$ will appear again from \Cref{sec local factors} onwards.
\subsection{Cyclicity}\label{sec cyclicity} Let $M:=C_c^{\infty}(G(\mathbf{Z}_p)\backslash G(F)/G(\mathcal{O}_F))$ be the space of compactly supported $\mathbf{C}$-valued functions, that are left $G(\mathbf{Z}_p)$-invariant and right $G(\mathcal{O}_F)$-invariant. We regard $M$ as a module over the Hecke algebra $\mathcal{H}_{G(F)}^\circ\otimes \mathcal{H}_{G(\mathbf{Q}_p)}^\circ$ via
\begin{align*}
    \left((\theta_1\otimes \theta_2)\cdot\xi\right)(x):=\int_{G(F)}\int_{G(\mathbf{Q}_p)} \theta_1(g)\theta_2(h)\textcolor{Black}{\xi}(h^{-1}xg)\ dg\ dh\textcolor{Black}{,\ \ \ \xi\in M.}
\end{align*}
where as usual, $dg$, resp. $dh$, are the normalized Haar measures on $G(F)$, resp. $G(\mathbf{Q}_p)$, giving $G(\mathcal{O}_F)$, resp. $G(\mathbf{Z}_p)$, volume $1$.

We wish to show that $M$ is cyclic \textcolor{Black}{over $\mathcal{H}_{G(F)}^\circ\otimes \mathcal{H}_{G(\mathbf{Q}_p)}^\circ$}, generated by $\ch(G(\mathcal{O}_F))$. Since $\mathrm{Res}_{F/\mathbf{Q}_p}\GL_2$ is non-split over $\mathbf{Q}_p$, we cannot apply the result of \cite{sakellaridis2013spherical}, like we did in \cite{groutides2024rankinselbergintegralstructureseuler}. Instead, we proceed via an argument inspired by the proofs of the uniqueness of Whittaker and Shintani functions of \cite{murase1996shintani} and \cite{kato2004p}.
\begin{thm}\label{thm cyclicity of M}
    The module $M$ is cyclic over $\mathcal{H}_{G(F)}^\circ\otimes \mathcal{H}_{G(\mathbf{Q}_p)}^\circ$ generated by the characteristic function $\ch(G(\mathcal{O}_F))$.
\end{thm}
\begin{proof}
     We have an action of $G(F)$ on $\mathbf{P}^1_F$ given by $\left[\begin{smallmatrix}
        a & b \\
        c & d
    \end{smallmatrix}\right]\cdot (x:y)=(ax+by:cx+dy)$.
    One easily checks that the flag variety decomposes into $G(\mathbf{Q}_p)$-orbits:
    \begin{align}\label{eq: decomp flag variety}
        G(F)/B(F)\simeq_\iota \mathbf{P}_F^1=G(\mathbf{Q}_p)\cdot (1:0)\ \sqcup\ G(\mathbf{Q}_p)\cdot (1:\alpha)
    \end{align}
    where $\iota$ is given by $g\mapsto g\cdot (1:0)$. The two orbits can be identified with:
    \begin{itemize}
        \item $G(\mathbf{Q}_p)\cdot (1:0)\simeq G(\mathbf{Q}_p)/B(\mathbf{Q}_p)=G(\mathbf{Z}_p)$\item $G(\mathbf{Q}_p)\cdot (1:\alpha)\simeq G(\mathbf{Q}_p)/\left(\mathrm{Res}_{F/\mathbf{Q}_p}\GL_1(\mathbf{Q}_p)\right)$ where $\mathrm{Res}_{F/\mathbf{Q}_p}\GL_1$ is regarded as a subgroup of $G$ via $a+b\alpha\mapsto \left[\begin{smallmatrix}
            a & b\\
            b\alpha^2 & a
        \end{smallmatrix}\right]$.
    \end{itemize}
    Write $t(a,b):=\mathrm{diag}(p^a,p^b)$ for $a,b\in\mathbf{Z}$. By \cite[A$.0.2$]{groutides2024integral} we have the decomposition 
    $$G(\mathbf{Q}_p)=\bigcup_{a\in\mathbf{Z}_\geq 0} G(\mathbf{Z}_p)\ t(a,\textcolor{Black}{0})^{-1}\ \mathrm{Res}_{F/\mathbf{Q}_p}\GL_1(\mathbf{Q}_p).$$
    Thus
    \begin{align*}G(\mathbf{Q}_p)\cdot (1:\alpha) = \bigcup_{a\in\mathbf{Z}_{\geq 0}} G(\mathbf{Z}_p)\ t(a,\textcolor{Black}{0})^{-1}\cdot (1:\alpha)
    =\bigcup_{a\in\mathbf{Z}_{\geq 0}} G(\mathbf{Z}_p)\ n_{p^a\alpha}\cdot (1:0)
    \end{align*}
    where $n_{p^a\alpha}:=\left[\begin{smallmatrix}
        1 & \\
        p^a\alpha & 1
    \end{smallmatrix}\right]$. Thus pulling back $\iota$ in \eqref{eq: decomp flag variety} we obtain $G(F)/B(F)=\overline{G(\mathbf{Z}_p)}\ \cup\ \left(\cup_{a\in\mathbf{Z}_{\geq 0}} \overline{G(\mathbf{Z}_p)n_{p^a\alpha}}\right)$ where the overline denotes the image in $G(F)/B(F)$ under the canonical projection. Iwasawa decomposition then implies that
    $$G(\mathcal{O}_F)=G(\mathbf{Z}_p)B(\mathcal{O}_F)\cup \left(\bigcup_{a\in\mathbf{Z}_{\geq 0}} G(\mathbf{Z}_p)\ n_{p^a\alpha}\ B(\mathcal{O}_F)\right).$$
    Cartan decomposition then gives
    \begin{align}\label{eq:decomp 1}
        G(F)=\bigcup_{\lambda_1\leq \lambda_2}\left(G(\mathbf{Z}_p)B(\mathcal{O}_F)\cup \left(\bigcup_{a\in\mathbf{Z}_{\geq 0}} G(\mathbf{Z}_p)\ n_{p^a\alpha}\ B(\mathcal{O}_F)\right)\right)\ t(\lambda_1,\lambda_2)\ G(\mathcal{O}_F).
    \end{align}
    Using that $B(\mathcal{O}_F)t(\lambda_1,\lambda_2)G(\mathcal{O}_F)=t(\lambda_1,\lambda_2)G(\mathcal{O}_F)$ for $\lambda_1\leq\lambda_2$ and $n_{p^a\alpha}t(\lambda_1,\lambda_2)G(\mathcal{O}_F)=t(\lambda_1,\lambda_2)G(\mathcal{O}_F)$ for $a\geq \lambda_2-\lambda_1$, we see that \eqref{eq:decomp 1} reduces to
    \begin{align}\label{eq:decomp 2}
        G(F)=\bigcup_{\mu_1\geq \mu_2}\ \ \bigcup_{0\leq a < \mu_1-\mu_2}\left(G(\mathbf{Z}_p)\ t(\mu_1,\mu_2)\ G(\mathcal{O}_F)\ \cup\ G(\mathbf{Z}_p)\ n_{p^a\alpha}^t\ t(\mu_1,\mu_2)\ G(\mathcal{O}_F)\right)
    \end{align}
    \textcolor{Black}{Indeed, after conjugating by the non-trivial Weyl element $w_0:=\left[\begin{smallmatrix}
        & -1\\
        1& 
    \end{smallmatrix}\right]\in G(\mathbf{Z}_p)$, we have $w_0t(\lambda_1,\lambda_2)w_0^{-1}=t(\lambda_2,\lambda_1)$ and $w_0n_{p^a\alpha}t(\lambda_1,\lambda_2)w_0^{-1}=n_{-p^a\alpha}^tt(\lambda_2,\lambda_1)=\left[\begin{smallmatrix}
        -1 & \\
         & 1
    \end{smallmatrix}\right]n_{p^a\alpha}^tt(\lambda_2,\lambda_1)\left[\begin{smallmatrix}
        -1 & \\
         & 1
    \end{smallmatrix}\right]$.} Now, $$n_{p^a\alpha}^t\  t(\mu_1,\mu_2)=t(a,\textcolor{Black}{0})\ n_\alpha^t\  t(\mu_1-a,\mu_2)$$
    and $\mu_1-\alpha-\mu_2>0$. Using this, we re-write \eqref{eq:decomp 2} as 
    \begin{align}\label{eq: Generalized Cartan decomp}
        G(F)= \bigcup_{\substack{\nu_1\geq \nu_2\\ \nu\geq0}} G(\mathbf{Z}_p)\ t(\textcolor{Black}{0},\nu)^{-1}\ n_\alpha^t\ t(\nu_2,\nu_1)^{-1}\ G(\mathcal{O}_F).
    \end{align}
    This is the generalized Cartan decomposition we will use in order to establish cyclicity of $M$. Because of \eqref{eq: Generalized Cartan decomp}, it suffices to show that the characteristic functions:
\begin{align}\label{eq: xi functions}
\xi_{\nu_2,\nu_1,\nu}:=\ch\left(G(\mathbf{Z}_p)\ t(\textcolor{Black}{0},\nu)^{-1}\ n_\alpha^t\ t(\nu_2,\nu_1)^{-1}\ G(\mathcal{O}_F)\right)\  \mathrm{with}\ \nu_1\geq \nu_2\ \mathrm{and}\ \nu\geq 0
\end{align}
 all belong in the submodule $M^{'}:=\left(\mathcal{H}_{G(F)}^\circ\otimes \mathcal{H}_{G(\mathbf{Q}_p)}^\circ\right)\cdot \ch(G(\mathcal{O}_F)).$ It will also be convenient to set:
 \begin{itemize}
     \item $\xi_0:=\xi_{0,0,0}.$
     \item $\phi_{a,b}^{\mathbf{Q}_p}:=\ch\left(G(\mathbf{Z}_p)\ t(a,b)^{-1}\ G(\mathbf{Z}_p)\right)\in \mathcal{H}_{G(\mathbf{Q}_p)}^\circ$ with $a,b\in\mathbf{Z}.$
     \item $\phi_{a,b}^F:=\ch\left(G(\mathcal{O}_F)\ t(a,b)\ G(\mathcal{O}_F)\right)\in \mathcal{H}_{G(F)}^\circ$ with $a,b\in\mathbf{Z}$.
 \end{itemize}
  We proceed by induction on the set of tuples of non-negative integers
 $$s(\nu_2,\nu_1,\nu):=(\nu+\nu_1-\nu_2,\nu_1-\nu_2)\in\mathbf{Z}_{\geq0}\times \mathbf{Z}_{\geq0}$$
 under lexicographic ordering.
 The base case is $s(\nu_2,\nu_1,\nu)=(0,0)$ in which case $\xi_{\nu_2,\nu_1,\nu}$ is nothing more than $\ch(G(\mathcal{O}_F))$ and there is nothing to prove. Now consider arbitrary $(\nu_2,\nu_1,\nu)$ as in \eqref{eq: xi functions}. The function
 \begin{align*}
     \left(\phi_{\nu_2,\nu_1}^F\otimes \phi_{1,\nu}^{\mathbf{Q}_p}\right)\cdot \xi_0=\int_{G(F)}\int_{G(\mathbf{Q}_p)}\phi_{1,\lambda}^{\mathbf{Q}_p}(h)\ \phi_{\nu_2,\nu_1}^F(g)\ \xi_0(h^{-1}(-)g)\ dg\ dh
 \end{align*}
 is supported on $U_{\nu_2,\nu_1,\nu}:=G(\mathbf{Z}_p)t(\textcolor{Black}{0},\nu)^{-1}G(\mathcal{O}_F)t(\nu_2,\nu_1)^{-1} G(\mathcal{O}_F)\supseteq G(\mathbf{Z}_p)t(\textcolor{Black}{0},\nu)^{-1\ }n_\alpha^t\ t(\nu_2,\nu_1)^{-1} G(\mathcal{O}_F) $ . Hence, up to a non-zero scalar, the function $\xi_{\nu_2,\nu_1,\nu}$ is given by a finite sum
 \begin{align}\label{eq: 7}
\xi_{\nu_2,\nu_1,\nu}=\left(\phi_{\nu_2,\nu_1}^F\otimes \phi_{1,\nu}^{\mathbf{Q}_p}\right)\cdot \xi_0 - \left(\sum_{(\Tilde{\nu}_2,\Tilde{\nu}_1,\Tilde{\nu})} c({\Tilde{\nu}_2,\Tilde{\nu}_1,\Tilde{\nu}})\xi_{\Tilde{\nu}_2,\Tilde{\nu}_1,\Tilde{\nu}}\right)
 \end{align}
 where $c({\Tilde{\nu}_2,\Tilde{\nu}_1,\Tilde{\nu}})$ are non-zero scalars and the sum on the right runs over a finite number of triples $({\Tilde{\nu}_2,\Tilde{\nu}_1,\Tilde{\nu}})$ for which $G(\mathbf{Z}_p)t(1,\Tilde{\nu})^{-1}\ n_\alpha^t\ t(\Tilde{\nu}_2,\Tilde{\nu}_1)^{-1}\ G(\mathcal{O}_F)\subseteq U_{\nu_2,\nu_1,\nu}$ and $\xi_{\Tilde{\nu}_2,\Tilde{\nu}_1,\Tilde{\nu}}\neq \xi_{\nu_2,\nu_1,\nu}$. Consider such a triple appearing in the sum of \eqref{eq: 7}. We can write 
 \begin{align}\label{eq: 8}t(1,\Tilde{\nu})^{-1}\ n_\alpha^t\ t(\Tilde{\nu}_2,\Tilde{\nu}_1)^{-1}= k\ t(\textcolor{Black}{0},\nu)^{-1}\ k_1\ t(\nu_2,\nu_1)^{-1}\ k_2 \end{align}
 where $k\in G(\mathbf{Z}_p)$ and $k_1,k_2\in G(\mathcal{O}_F)$. Applying $v_p(\det(-))$ to \eqref{eq: 8}, we obtain 
 \begin{align}\label{eq: 9}
\Tilde{\nu}_2+\Tilde{\nu}_1+\Tilde{\nu}=\nu_2+\nu_1+\nu.
 \end{align}
 An explicit calculation \textcolor{Black}{of \eqref{eq: 8}} shows that $$t(1,\Tilde{\nu})^{-1}\ n_\alpha^t\ t(\Tilde{\nu}_2,\Tilde{\nu}_1)^{-1}=\left[\begin{smallmatrix}
     p^{-\Tilde{\nu}_2} & \alpha p^{-\Tilde{\nu}_1}\\
      & p^{-\Tilde{\nu}-\Tilde{\nu}_1}
 \end{smallmatrix}\right]\textcolor{Black}{=\left[\begin{smallmatrix}
     * & *\\
     * & r_1p^{-\nu_2}+r_2p^{-\nu-\nu_2}+r_3p^{-\nu_1}+r_4p^{-\nu-\nu_1}
 \end{smallmatrix}\right],\ \ \ r_i\in\mathbf{Z}_p}$$
 Thus, taking $p$-adic valuation of the bottom-right entry of \eqref{eq: 8}, we get
 \begin{align}\label{eq: 10}
     \Tilde{\nu}+\Tilde{\nu}_1\leq \nu+\nu_1
 \end{align}
 Combining \eqref{eq: 9} and \eqref{eq: 10}, we obtain
 \begin{align}\label{eq:11}
     \Tilde{\nu}+\Tilde{\nu}_1-\Tilde{\nu}_2\leq \nu+\nu_1-\nu_2
 \end{align}
 If the inequality in \eqref{eq:11} is strict then, the function $\xi_{\Tilde{\nu}_2,\Tilde{\nu}_1,\Tilde{\nu}}$ is contained in $M^{'}$ by induction. Thus, we assume that the inequality in \eqref{eq:11} is an equality, in which case combining again with \eqref{eq: 9}, we obtain $\Tilde{\nu}_2=\nu_2$. Now, writing $k^{-1}=\left[\begin{smallmatrix}
     h_1 & h_2 \\
     * & *
 \end{smallmatrix}\right]$, we can re-write \eqref{eq: 8} as 
 \begin{align}\label{eq:12}
     \left[\begin{matrix}
         * & h_2p^{-\Tilde{\nu}-\Tilde{\nu}_1}+\alpha h_1p^{-\Tilde{\nu}_1}\\
         * & *
     \end{matrix}\right]=t(\textcolor{Black}{0},\nu)^{-1}k_1t(\nu_2,\nu_1)^{-1}k_2.
 \end{align}
 We can re-write the top-right entry of \eqref{eq:12} as
 \begin{align}\label{eq: 13}
     h_2p^{-\Tilde{\nu}-\Tilde{\nu}_1}+\alpha h_1p^{-\Tilde{\nu}_1}=p^C\left(up^{-\Tilde{\nu}-\Tilde{\nu}_1+v_p(h_2)-C}+\alpha h_1p^{-\Tilde{\nu}_1-C}\right)
 \end{align}
 where $C:=\mathrm{min}\{-\Tilde{\nu}_1,-\Tilde{\nu}_1-\Tilde{\nu}+v_p(h_2)\}$ and $u\in\mathbf{Z}_p^\times$. It's not hard to see, using $\mathcal{O}_F=\mathbf{Z}_p\oplus\alpha\mathbf{Z}_p$ and $k\in G(\mathbf{Z}_p)$, that for both values of $C$, the bracketed element of \eqref{eq: 13} is always in $\mathcal{O}_F^\times$ and thus the $p$-adic valuation of \eqref{eq: 13} is equal to $C$. Going back to \eqref{eq:12} and taking $p$-adic valuation of the top-right entry, we obtain
 \begin{align}\label{eq: 14}
     C\geq -\nu_1.
 \end{align}
 Again, for both possible values of $C$, this implies that $\Tilde{\nu}_1\leq \nu_1$. Since we are still assuming that \eqref{eq:11} is an equality, hence  $\Tilde{\nu}_2=\nu_2$, we have \begin{align}
\Tilde{\nu}_1-\Tilde{\nu}_2\leq \nu_1-\nu_2.
 \end{align}
 However this inequality must now be strict since we are assuming that $\xi_{\Tilde{\nu}_2,\Tilde{\nu}_1,\Tilde{\nu}}\neq \xi_{\nu_2,\nu_1,\nu}$. Thus, in any case, we have
 $$s(\Tilde{\nu}_2,\Tilde{\nu}_1,\Tilde{\nu})<s(\nu_2,\nu_1,\nu)$$
 and the result follows by induction.
\end{proof}

\begin{cor}\label{cor cyclicity of I}
    The module $\mathcal{I}(G(F)/G(\mathcal{O}_F))$ is cyclic over $\mathcal{H}_{G(F)}^\circ$, generated by $\ch(\mathbf{Z}_p^2)\otimes \ch(G(\mathcal{O}_F))$.
    \begin{proof}
        The detailed argument of \cite[Theorem $3.2.1$]{groutides2024rankinselbergintegralstructureseuler} works verbatim. The key ingredients in this case being \Cref{thm cyclicity of M} and \cite[ Lemma $4.3.3$, Proposition $4.3.4$, Theorem $4.3.6$]{Loeffler_2021}.
    \end{proof}
\end{cor}

\subsection{Utilizing the mirabolic subgroup}
Let $P$ denote the mirabolic subgroup of $G$. In this short section, we verify that the passage from $\mathcal{I}(G(F)/G(\mathcal{O}_F))$ to $C_c^\infty(P(\mathbf{Q}_p)\backslash G(F)/G(\mathcal{O}_F))$ still works (integrally) as in \cite[\S 3.3 \& \S 3.4]{groutides2024rankinselbergintegralstructureseuler}. This is a crucial step in our argument, since it is what allows us to attack the problem in its full generality, at an integral level. Here as usual, $C_c^\infty(P(\mathbf{Q}_p)\backslash G(F)/G(\mathcal{O}_F))$ denotes the space of left $P(\mathbf{Q}_p)$-invariant and right $G(\mathcal{O}_F)$-invariant $\mathbf{C}$-valued function on $G(F)$, that are compactly supported modulo $P(\mathbf{Q}_p)$, regarded as a $\mathcal{H}_{G(F)}^\circ$-module under the right translation action. \textcolor{Black}{Explicitly we have}
$$\textcolor{Black}{(f\cdot\xi)(x):=\int_{G(F)}f(g)\xi(xg)\ dg,\ f\in \mathcal{H}_{G(F)}^\circ,\xi\in C_c^\infty(P(\mathbf{Q}_p)\backslash G(F)/G(\mathcal{O}_F)), x\in G(F).}$$We also write $C_c^\infty(P(\mathbf{Q}_p)\backslash G(F)/G(\mathcal{O}_F),\mathbf{Z}[1/p])$ for the lattice of such functions which take values in $\mathbf{Z}[1/p]$.
\begin{prop}\label{prop Hecke equiv map}
    There exists a $\mathcal{H}_{G(F)}^\circ$-equivariant map $\Xi_c$ which induces the following diagram
    \[\begin{tikzcd}[ampersand replacement=\&,cramped]
	{\mathcal{I}(G(F)/G(\mathcal{O}_F))} \& {C_c^\infty(P(\mathbf{Q}_p)\backslash G(F)/G(\mathcal{O}_F))} \& {} \\
	{\mathcal{I}(G(F)/G(\mathcal{O}_F),\mathbf{Z}[1/p])} \& {C_c^\infty(P(\mathbf{Q}_p)\backslash G(F)/G(\mathcal{O}_F),\mathbf{Z}[1/p])}
	\arrow["{\Xi_c}", from=1-1, to=1-2]
	\arrow[hook', from=2-1, to=1-1]
	\arrow["{\Xi_c}", from=2-1, to=2-2]
	\arrow[hook, from=2-2, to=1-2]
\end{tikzcd}\]
and maps $\ch(\mathbf{Z}_p^2)\otimes \ch(G(\mathcal{O}_F))$ to $\ch(P(\mathbf{Q}_p)G(\mathcal{O}_F))$.
\begin{proof}
    The construction in \cite[\S$3.3$, \S$3.4$]{groutides2024rankinselbergintegralstructureseuler} works verbatim, this time utilizing \Cref{cor cyclicity of I} in order to prove the equivalent statement of \cite[Proposition $3.4.2$]{groutides2024rankinselbergintegralstructureseuler}. \textcolor{Black}{The map $\Xi_c$ is given explicitely by firstly applying the map $\Xi$ which sends an element $\phi\otimes\xi$ to the function $g\mapsto \int_{G(\mathbf{Q}_p)}\xi(h^{-1}g)\phi((0,1)h)\ dh$ (that, a priori, might not be compactly supported) and then applying the Hecke operator $(1-\mathcal{S}_F^{-1})$ where $\mathcal{S}_F:=\ch(\left[\begin{smallmatrix}
        p & \\
        & p
    \end{smallmatrix}\right]G(\mathcal{O}_F)).$ Concretely, we have $\Xi_c(\phi\otimes\xi):=(1-\mathcal{S}_F^{-1})\cdot \Xi(\phi\otimes\xi).$} There's one small thing that needs to be checked, and that is the convergence of the integral involved in the construction of the map $\Xi_c$. This does depend on the particular setup. In this case, the convergence follows from the fact that the $G(\mathbf{Q}_p)\cap gG(\mathcal{O}_F)g^{-1}$ is an open compact subgroup of $G(\mathbf{Q}_p)$, which is easily seen to be true since $F/\mathbf{Q}_p$ is unramified.
    \end{proof}
\end{prop}

   We write $\delta_P$ for the modular quasi-character of the mirabolic, and \textcolor{Black}{under the same action as the one of $\mathcal{H}_{G(F)}^\circ$ on $C_c^\infty(P(\mathbf{Q}_p)\backslash G(F)/G(\mathcal{O}_F))$,} we consider the $\delta_P$-twisted $\mathcal{H}_{G(F)}^\circ$-module of coinvariants:
  $$C_c^\infty(G(F)/G(\mathcal{O}_F))_{P(\mathbf{Q}_p),\delta_P}:=C_c^\infty(G(F)/G(\mathcal{O}_F))/\left\langle x\cdot \xi -\delta_P(x)\xi\ |\ x\in P(\mathbf{Q}_p),\ \xi\in C_c^\infty(G(F)/G(\mathcal{O}_F))\right\rangle.$$

  \begin{prop}
  There exists a $\mathcal{H}_{G(F)}^\circ$-equivariant isomorphism,
 \begin{align}\label{eq: map Phi_c}
     \Phi_c:C_c^\infty(P(\mathbf{Q}_p)\backslash G(F)/G(\mathcal{O}_F))&\overset{\simeq}{\longrightarrow} C_c^\infty(G(F)/G(\mathcal{O}_F))_{P(\mathbf{Q}_p),\delta_P}\\
    \nonumber \ch(P(\mathbf{Q}_p)gG(\mathcal{O}_F))&\mapsto\frac{1}{\mathrm{vol}_{P(\mathbf{Q}_p)}(P(\mathbf{Q}_p)\cap g G(\mathcal{O}_F)g^{-1})}\ch(g G(\mathcal{O}_F))
 \end{align}
 \begin{proof}
     This follows in the same manner as \cite[Proposition $6.1.1$]{groutides2024rankinselbergintegralstructureseuler}. \textcolor{Black}{One defines an inverse morphism to $\Phi_c$ by mapping a function $\xi\in C_c^\infty(G(F)/G(\mathcal{O}_F))$ to the function $g\mapsto\int_{P(F)}\xi(xg)\ d^Rx$ where $d^Rx$ is the normalized right Haar measure on $P(F)$. This factors through the twisted coinvariants and is $\mathcal{H}_{G(F)}^\circ$-equivariant by construction.} The measure $\mathrm{vol}_{P(\mathbf{Q}_p)}$ is the normalized (left or right) Haar measure on $P(\mathbf{Q}_p)$ which gives $P(\mathbf{Z}_p)$ volume $1$ and the construction of $\Phi_c$ is independent of the choice between left or right.
 \end{proof}
 \end{prop}
\noindent It essentially follows from Iwasawa decomposition that 
 \begin{align}\label{eq: 17}
     G(F)=\bigcup_{\substack{a\in\mathbf{Z}\\ b\in\mathbf{Z}_{\geq0}}} P(\mathbf{Q}_p)\left[\begin{smallmatrix}
         p^a & \\
         & p^a
     \end{smallmatrix}\right]\left[\begin{smallmatrix}
         1 & \alpha p^{-b}\\
         & 1
     \end{smallmatrix}\right]G(\mathcal{O}_F)
 \end{align}
 where the union is in fact disjoint. Using the notation of \Cref{thm cyclicity of M} we had $t(a,a)=\left[\begin{smallmatrix}
         p^a & \\
         & p^a
     \end{smallmatrix}\right]$ and $n_{p^{-b}\alpha}^t= \left[\begin{smallmatrix}
         1 & \alpha p^{-b}\\
         & 1
     \end{smallmatrix}\right]$. To make matters simpler, we write 
     \begin{align*}
         \mathbf{t}_a&:=t(a,a), \ a\in\mathbf{Z}\\
         \mathbf{n}_b&:=n_{p^{-b}\alpha}^t,\  b\in\mathbf{Z}_{\geq 0}.
     \end{align*}
     The functions $\ch(P(\mathbf{Q}_p) \mathbf{t}_a\mathbf{n}_b G(\mathcal{O}_F))$ form a basis for $C_c^\infty(P(\mathbf{Q}_p)\backslash G(F)/G(\mathcal{O}_F))$ and it is not hard to calculate that
     \begin{align}\label{eq:values of Phi_c}\Phi_c\left(\ch(P(\mathbf{Q}_p) \mathbf{t}_a\mathbf{n}_b G(\mathcal{O}_F))\right)=\begin{dcases}
         \ch(\mathbf{t}_a G(\mathcal{O}_F)),\ &\mathrm{if}\ b=0\\
         (p-1)p^{b-1}\ch(\mathbf{t}_a\mathbf{n}_b G(\mathcal{O}_F)),\ &\mathrm{if}\ b\in\mathbf{Z}_{>0}.
     \end{dcases}
     \end{align}
\subsection{Local zeta integrals and Asai periods}\label{sec zeta integrals}
Let $\Pi_F$ be an unramified \textit{Whittaker type} representation of $G(F)$. In other words, $\Pi_F$ is an unramified principal-series that is not a twist of the degenerate principal-series which admits the Steinberg representation as a subquotient. For more details on these representations, we refer the reader to \cite{gel1972representation}, \cite{casselman1980unramified} and \cite{bump_1997}. 

  We fix once and for all the following additive characters: We write $\psi:\mathbf{Q}_p\rightarrow \mathbf{C}^\times$ for the standard additive character of conductor $\mathbf{Z}_p$ and set $\psi_F:F\rightarrow\mathbf{C}^\times, x\mapsto\psi(\mathrm{Tr}_{F/\mathbf{Q}_p}(x\alpha))$. Then $\psi_F$ has conductor $\mathcal{O}_F$ and is trivial on $\mathbf{Q}_p$. We will identify a representation $\Pi_F$ as above with its Whittaker model $\mathcal{W}(\Pi_F,\psi_F)$\textcolor{Black}{, which is a space of functions on $G(F)$ transforming by $\psi_F$ under left $N(F)$-translations. This space is regarded as a $G(F)$-representation under right translations}. \textcolor{Black}{We} write $W_{\Pi_F}^\mathrm{sph}$ for the normalized spherical vector (i.e. $W_{\Pi_F}^\mathrm{sph}(1)=1$).
\begin{defn}[\cite{flicker1988twisted}]\label{def Asai zeta integral}
For $\phi\in\mathcal{S}(\mathbf{Q}_p^2)$ and $W\in\mathcal{W}(\Pi_F,\psi_F)$, we set
\begin{align*}
    Z(\phi,W,s):=\int_{N(\mathbf{Q}_p)\backslash G(\mathbf{Q}_p)}W(g)\phi((0,1)g)|\det(g)|_p^s\ dg
\end{align*}
\end{defn}
\begin{rem}
    The zeta integral $Z(\phi,W,s)$ is well-defined by our choice of additive character $\psi_F$, and it is the Asai analogue (for inert primes) of the local Rankin-Selberg zeta integral of \cite{jacquet1983rankin}.
\end{rem}
\begin{prop}
\begin{enumerate}
    \item The zeta integral $Z(\phi,W,s)$ converges absolutely for $\Re(s)$ large enough (independently of $\phi$ and $W$) and admits unique meromorphic continuation as a rational function of $p^s$.
    \item As $\phi$ and $W$ vary, the $Z(\phi,W,s)$ generate the fractional ideal $\L^\mathrm{As}(\Pi_F,s)\mathbf{C}[p^s,p^{-s}]$.
    \end{enumerate}
    \begin{proof}
        This is part of \cite{flicker1993zeroes} (see also \cite[\S $2$]{matringe2009conjecturesdistinctionasailfunctions}).
    \end{proof}
\end{prop}
\begin{lem}\label{lem props of Asai integral}
    \begin{enumerate}
        \item The local Asai $L$-factor $L^\mathrm{As}(\Pi_F,s)$ is given by $L(\Pi_F,s)L(\omega_{\Pi_F},2s)$ where $L(\Pi_F,s)$ is the standard Jacquet-Langlands $L$-factor attached to $\Pi_F$ (\cite[\S $1.3$]{jacquet2006automorphic} ) and $L(\omega_{\Pi_F},2s)$ is the usual Tate $L$-factor associated to the central character of $\Pi_F$ \textcolor{Black}{denoted by $\omega_{\Pi_F}$}. 
        \item We have $Z(\ch(\mathbf{Z}_p^2),W_{\Pi_F}^\mathrm{sph},s)=L^\mathrm{As}(\Pi_F,s)$.
        \end{enumerate}
        \begin{proof}
            The first part of the proposition follows from \cite{matringe2008distinguishedrepresentationsexceptionalpoles} (see also \cite{matringe2009conjecturesdistinctionasailfunctions} for the reducible cases) keeping in mind that $\Pi_F$ is unramified of Whittaker type. The second part can be found in \cite{8b685a57-0995-30f2-97f8-c2f88420d820}.
        \end{proof}
\end{lem}
\noindent It will be very convenient for us to also define a secondary zeta integral which also appears in \cite[\S $2$]{matringe2009conjecturesdistinctionasailfunctions} and \cite[\S $2.2$]{grossi2020norm}.
\begin{defn}
    For $W\in\mathcal{W}(\Pi_F,\psi_F)$ we set
    \begin{align*}
        \Psi(W,s):=\int_{\mathbf{Q}_p^\times} W\left(\left[\begin{smallmatrix}
            x & \\
            & 1
        \end{smallmatrix}\right]\right)|x|_p^{s-1}\ d^\times x
    \end{align*}
    where $d^\times x$ is the normalized Haar measure on $\mathbf{Q}_p^\times$ which gives $\mathbf{Z}_p^\times$ volume $1$.
\end{defn}
\noindent By a standard argument, this also converges absolutely (independently of $W$) for $\Re(s)$ large enough, and as we vary $W$, $\Psi(W,s)$ is still an element of the fractional ideal $L^\mathrm{As}(\Pi_F,s)\mathbf{C}[p^s,p^{-s}]$. Additionally, as in \cite[\S $2.2$]{grossi2020norm}, we have 
\begin{align}\label{eq: L on spherical vector}\Psi(W_{\Pi_F}^\mathrm{sph},s)=\frac{L^\mathrm{As}(\Pi_F,s)}{L(\omega_{\Pi_F},2s)}.
\end{align}

\begin{defn}\label{def periods}Let $\Pi_F$ be an unramified Whittaker type representation of $G(F)$.
\begin{enumerate} 
    \item For $\phi\in\mathcal{S}(\mathbf{Q}_p^2)$ and $W\in\Pi_F$, we set
    $$\mathcal{Z}(\phi\otimes W):=\lim_{s\rightarrow 0} \frac{Z(\phi,W,s)}{L^\mathrm{As}(\Pi_F,s)}.$$
    \item For $W\in\Pi_F$, we set
    $$ \mathscr{L}(W):= \lim_{s\rightarrow 0} \frac{\Psi(W,s)}{L^\mathrm{As}(\Pi_F,s)}.$$
\end{enumerate}
\end{defn}

\begin{lem}\label{lem non-zero Asai periods}Let $\Pi_F$ be an unramified Whittaker type representation of $G(F)$.
    \begin{enumerate}
        \item The linear form $\mathcal{Z}$ gives a non-zero period in $\mathrm{Hom}_{G(\mathbf{Q}_p)}(\mathcal{S}(\mathbf{Q}_p^2)\otimes \Pi_F,\mathbf{1})$ which takes the value $1$ on the unramified vector $\ch(\mathbf{Z}_p^2)\otimes W_{\Pi_F}^\mathrm{sph}.$
        \item Suppose that $\Pi_F$ has non-trivial central character. Then the linear form $\mathscr{L}$ gives a non-zero period in $\mathrm{Hom}_{P(\mathbf{Q}_p)}(\Pi_F,\delta_P)$ which takes the value $L(\omega_{\Pi_F},0)^{-1}$ on the unramified vector $W_{\Pi_F}^\mathrm{sph}$.
    \end{enumerate}
    \begin{proof}
        This follows from a straightforward computation using \eqref{eq: L on spherical vector} and \Cref{lem props of Asai integral}. See also \cite[\S $2$]{matringe2008distinguishedrepresentationsexceptionalpoles}.
    \end{proof}
\end{lem}
\noindent We can now deduce that the module $\mathcal{I}(G(F)/G(\mathcal{O}_F))$ which was shown to be cyclic over $\mathcal{H}_{G(F)}^\circ$ in \Cref{thm cyclicity of M}, is in fact \textit{free}, and the Hecke equivariant map $\Xi_c$ of \Cref{prop Hecke equiv map} is injective. This follows from \Cref{lem non-zero Asai periods}(1), \cite[Proposition $4.4.1$]{Loeffler_2021} \textcolor{Black}{(whose proof is formal and the statement remains unchanged when} translated to this setup), and the fact that the family of representations $\Pi_F$ considered here is dense in $\mathrm{Spec}(\mathcal{H}_{G(F)}^\circ)$. \textcolor{Black}{We provide a few more details on how to deduce freeness. Firstly, given a representation $\Pi_F$ as in \Cref{lem non-zero Asai periods}, we write $\Theta_{\Pi_F}:\mathcal{H}_{G(F)}^\circ\rightarrow\mathbf{C}$ for its spherical Hecke eigensystem. Suppose that $\mathcal{P}\cdot\left(\ch(\mathbf{Z}_p^2)\otimes \ch(G(\mathcal{O}_F))\right)=0$ in $\mathcal{I}(G(F)/G(\mathcal{O}_F))$ for some Hecke operator $\mathcal{P}\in\mathcal{H}_{G(F)}^\circ$. By \cite[Proposition $4.4.1$]{Loeffler_2021} and \Cref{lem non-zero Asai periods} we have $0=\Theta_{\Pi_F}(\mathcal{P}^{'})\cdot \mathcal{Z}(\ch(\mathbf{Z}_p^2)\otimes W_{\Pi_F}^\mathrm{sph})=\Theta_{\Pi_F}(\mathcal{P}^{'})$. This holds for all such $\Pi_F$ and thus by density, $\mathcal{P}^{'}$ is identically zero. Recall that $\mathcal{P}^{'}$ is defined by $\mathcal{P}^{'}(g)=\mathcal{P}(g^{-1})$, and so $\mathcal{P}$ is also identically zero. This shows that $\mathcal{I}(G(F)/G(\mathcal{O}_F))$ has no annihilator under the Hecke action and hence it is free of rank one}. Thus, it is natural to make the following definition:
\begin{defn}
    Let $\delta\in\mathcal{I}(G(F)/G(\mathcal{O}_F))$. The Hecke operator $\mathcal{P}_\delta$ is the unique element of the spherical Hecke algebra $\mathcal{H}_{G(F)}^\circ$ which satisfies
    $$\mathcal{P}_\delta\cdot \left(\ch(\mathbf{Z}_p^2)\otimes \ch(G(\mathcal{O}_F))\right)=\delta.$$
\end{defn}
\begin{rem}
    Hecke operators of this form (with Frobenius twists) for integral elements $\delta$ with determinant level, will turn out to be precisely the local factors appearing in tame norm relations between motivic Asai-Flach classes, at inert primes $p$ in $E$.
\end{rem}
\begin{lem}\label{lem Asai Euler factor}
    There exists a polynomial $\mathcal{P}_{F,\mathrm{As}}(X)$ in $\mathcal{H}_{G(F)}^\circ(\mathbf{Z}[1/p])[X]$ such that
    $$\Theta_{\Pi_F}\left(\mathcal{P}_{F,\mathrm{As}}\right)(p^{-s})=L^\mathrm{As}(\Pi_F,s)^{-1}$$
    for every unramified Whittaker type $\Pi_F$. \textcolor{Black}{Once again, $\Theta_{\Pi_F}$ denotes the spherical Hecke eigensystem of $\Pi_F$.}
    \begin{proof}
        The polynomial is given explicitly by 
        $
            \mathcal{P}_{F,\mathrm{As}}(X)=\left(1-\tfrac{1}{p}\mathcal{T}_FX+\mathcal{S}_FX^2\right)\left(1-\mathcal{S}_FX^2\right)
        $
where $\mathcal{T}_F:=\ch(G(\mathcal{O}_F)\left[\begin{smallmatrix}
   p  & \\
     & 1
\end{smallmatrix}\right] G(\mathcal{O}_F))$ and $\mathcal{S}_F:=\ch(\left[\begin{smallmatrix}
    p & \\
    & p
\end{smallmatrix}\right] G(\mathcal{O}_F))$. The result then follows by a direct computation using well-known formulas for the spherical Hecke eigenvalues of these two operators.
    \end{proof}
\end{lem}

\subsubsection{Back to Hecke modules}\label{sec back to hecke modules} For each unramified Whittaker type representation $\Pi_F$ of $G(F)$ with non-trivial central character, we can define a linear form
\begin{align}\label{eq: map Lambda}
\Lambda_{\Pi_F}:C_c^\infty(G(F)/G(\mathcal{O}_F))_{P(\mathbf{Q}_p),\delta_P}\longrightarrow \mathbf{C},\ \  \Lambda_{\Pi_F}(\xi):=\mathscr{L}\left(\xi\cdot W_{\Pi_F}^\mathrm{sph}\right).
\end{align}
By \Cref{lem non-zero Asai periods}(2) this is well-defined, non-zero, and satisfies $\Lambda_{\Pi_F}(\ch(G(\mathcal{O}_F)))=L(\omega_{\Pi_F},0)^{-1}.$ Our goal is to now obtain explicit formulas in terms of Satake parameters and local Asai $L$-factors, for the forms $\Lambda_{\Pi_F}$ evaluated on elements as in \eqref{eq:values of Phi_c}: It suffices to evaluate $\Psi(\mathbf{t}_a\mathbf{n}_bW_{\Pi_F}^\mathrm{sph},s)$ for $a\in\mathbf{Z}$ and $b\in\mathbf{Z}_{\geq 0}$. Plugging everything in and using the Whittaker model $\mathcal{W}(\Pi_F,\psi_F)$, we get
\begin{align}\label{eq: explicit formula for psi}
    \Psi(\mathbf{t}_a\mathbf{n}_bW_{\Pi_F}^\mathrm{sph},s)&=\omega_{\Pi_p}(p)^a\int_{\mathbf{Q}_p^\times}\psi_F\left(xp^{-b}\alpha\right) W_{\Pi_F}^\mathrm{sph}\left(\left[\begin{smallmatrix}
        x & \\
        & 1
    \end{smallmatrix}\right]\right)|x|_p^{s-1}\ d^\times x\\
   \nonumber &=\omega_{\Pi_F}(p)^a\sum_{n\in\mathbf{Z}_{\geq 0}}\left(\int_{p^n\mathbf{Z}_p^\times}\psi_F\left(xp^{-b}\alpha\right)\ d^\times x\right)W_{\Pi_F}^\mathrm{sph}\left(\left[\begin{smallmatrix}
        p^n & \\
        & 1
    \end{smallmatrix}\right]\right)p^{n(1-s)}
\end{align}
The additive character $\psi_F((-)\alpha)$ restricted to $\mathbf{Q}_p$ is of conductor $\mathbf{Z}_p$ by construction. Thus, we can apply the same treatment to \eqref{eq: explicit formula for psi} as that of \cite[\S $6.3.1$]{groutides2024integral} \textcolor{Black}{and split the above expression into an $L$-factor term and a finite ``error'' term}. After some manipulations, and using \cite{shintani1976explicit} and \Cref{lem non-zero Asai periods}($2$), we see that \eqref{eq: explicit formula for psi} is given by
\begin{align}\label{eq: 22}
\Psi(\mathbf{t}_a\mathbf{n}_bW_{\Pi_F}^\mathrm{sph},s)=\omega_{\Pi_F}(p)^a\left(\sum_{n=0}^{b-1}\epsilon_n(b)\mathfrak{s}_n\left(\alpha_{\Pi_F},\beta_{\Pi_F}\right)p^{-ns}\right)+\omega_{\Pi_F}(p)^a\frac{L^\mathrm{As}(\Pi_F,s)}{L(\omega_{\Pi_F},2s)}.
\end{align}
where the sum is taken to be zero by convention if $b=0$. The notation is that of \cite{groutides2024rankinselbergintegralstructureseuler}:
\begin{align*}
    \epsilon_n(b):=\begin{dcases}
        -1,\ &\mathrm{if}\ 0\leq n<b-1\\
        \tfrac{-p}{p-1},\ &\mathrm{if}\ n=b
    \end{dcases}\ \ , \ \ \mathfrak{s}_n(x,y):=\frac{x^{n+1}-y^{n+1}}{x-y}
\end{align*}
and $\alpha_{\Pi_F},\beta_{\Pi_F}$ denote the Satake parameters of $\Pi_F$. Putting a few more things together, we obtain:
\begin{prop}\label{prop explicit Lambda}
Let $\Pi_F$ be unramified of Whittaker type with non-trivial central character. The linear form $\Lambda_{\Pi_F}$ is given explicitly by
\begin{align*}
\Lambda_{\Pi_F}\left(\ch(\mathbf{t}_a\mathbf{n}_bG(\mathcal{O}_F))\right)=\Theta_{\Pi_F}\left(\mathcal{S}_F^a\left(\sum_{n=0}^{b-1}\epsilon_n(b)\mathfrak{s}_n^\circ(\mathcal{S}_F,\mathcal{T}_F)\right)\mathcal{P}_{F,\mathrm{As}}(1) +\mathcal{S}_F^a(1-\mathcal{S}_F)\right).
\end{align*}
where $\mathfrak{s}_n^\circ(X,Y)\in\mathcal{H}_{G(F)}^\circ(\mathbf{Z}[1/p])[X,Y]$ is uniquely determined by imposing the condition $\Theta_{\Pi_F}(\mathfrak{s}_n^\circ(\mathcal{S}_F,\mathcal{T}_F))=\mathfrak{s}_n\left(\alpha_{\Pi_F},\beta_{\Pi_F}\right)$ for every such $\Pi_F$. Again the sum is taken to be zero by convention whenever $b=0$.
\begin{proof}
    From \eqref{eq: 22}, we immediately obtain an expression for $\mathscr{L}(\mathbf{t}_a\mathbf{n}_b W_{\Pi_F}^\mathrm{sph})$. \textcolor{Black}{The polynomials $\mathfrak{s}_n^\circ(X,Y)$ clearly exist and satisfy the property mentioned, since the expressions $\mathfrak{s}_n(\alpha_{\Pi_F},\beta_{\Pi_F})$ are symmetric polynomials in the Satake parameters, and $\Theta_{\Pi_F}(\mathcal{T}_F)=p(\alpha_{\Pi_F}+\beta_{\Pi_F}),\Theta_{\Pi_F}(\mathcal{S}_F)=\alpha_{\Pi_F}\beta_{\Pi_F}$.} Using \Cref{lem Asai Euler factor} and pulling back to the Hecke algebra through the spherical Hecke eigensystem $\Theta_{\Pi_F}$, we obtain the result.
\end{proof}
\end{prop}
\subsection{Local factors}\label{sec local factors} We now have all the required components to prove the analogue of \cite[Theorem $5.1.1$]{groutides2024rankinselbergintegralstructureseuler}. However, we will not be concerned with equivariant maps here since that part remains practically unchanged. For more details on that, see \textit{op.cit} and \cite[\S $3$]{Loeffler_2021}. \textcolor{Black}{Before stating the main result of this section, we once again recall that $\mathcal{P}^{'}(g):=\mathcal{P}(g^{-1})$ for any spherical Hecke operator $\mathcal{P}$, the open compact determinant level subgroup $G(\mathcal{O}_F)[p]$ defined at the beginning of \Cref{sec local results}, and the trace map $\mathrm{Tr}=\mathrm{Tr}^{G(\mathcal{O}_F)[p]}_{G(\mathcal{O}_F)}$ described at the end of \Cref{sec general notation}.}
 
\begin{thm}\label{thm integral local factors}
    \begin{enumerate}
        \item Let $\delta$ be an element of the integral lattice $\mathcal{I}(G(F)/G(\mathcal{O}_F),\mathbf{Z}[1/p])$. Then, the corresponding local factor $\mathcal{P}_\delta$ lies in $\mathcal{H}_{G(F)}^\circ(\mathbf{Z}[1/p])$.
        \item Let $\delta$ be an element of the determinant level integral lattice $\mathcal{I}_0(G(F)/G(\mathcal{O}_F)[p],\mathbf{Z}[1/p])$. Then, the corresponding local factor $\mathcal{P}_{\mathrm{Tr}(\delta)}$ lies in the ideal
        $$\left\langle (p-1)(1-\mathcal{S}_F),\mathcal{P}_{F,\mathrm{As}}^{'}(1)\right\rangle\subseteq\mathcal{H}_{G(F)}^\circ(\mathbf{Z}[1/p]).$$
        \item Finally, if $\delta$ is an element of the determinant level integral lattice $\mathcal{I}(G(F)/G(\mathcal{O}_F)[p],\mathbf{Z}[1/p])$, then corresponding local factor $\mathcal{P}_{\mathrm{Tr}(\delta)}$ lies in the ideal
        $$\left\langle p-1, \mathcal{P}_F^{'}(1)\right\rangle\subseteq\mathcal{H}_{G(F)}^\circ(\mathbf{Z}[1/p]).$$
        where $\mathcal{P}_F(X)\in\mathcal{H}_{G(F)}^\circ(\mathbf{Z}[1/p])[X]$ is the polynomial interpolating the usual Jacquet-Langlands local $L$-factors of unramified principal-series of $G(F)$, i.e. $\mathcal{P}_F(1)=\frac{\mathcal{P}_{F,\mathrm{As}}(1)}{1-\mathcal{S}_F}$.
    \end{enumerate}
    \begin{proof}
        We once again follow the approach in \cite{groutides2024rankinselbergintegralstructureseuler}. Let $\delta$ be as in the first part. We know from \Cref{prop Hecke equiv map} that the image of $\delta$ under $\Xi_c$ is $\mathbf{Z}[1/p]$-valued. It is thus a finite $\mathbf{Z}[1/p]$-linear combination of characteristic functions
        \begin{align}
    \ch(P(\mathbf{Q}_p)\mathbf{t}_a\mathbf{n}_b G(\mathcal{O}_F)),\  a\in\mathbf{Z},\ b\in\mathbf{Z}_{\geq 0}.
        \end{align}
        Further applying $\Phi_c$ and using \eqref{eq:values of Phi_c}, this becomes a finite $\mathbf{Z}[1/p]$ linear combination of elements of the form 
        $$\begin{dcases}
            \ch(\mathbf{t}_a G(\mathcal{O}_F)),\ &\mathrm{if}\ b=0\\
            (p-1)p^{b-1}\ch(\mathbf{t}_a\mathbf{n}_b G(\mathcal{O}_F)),\ &\mathrm{if}\ b>0.
        \end{dcases}$$
        Let $\Pi_F$ be an unramified Whittaker type representation of $G(F)$. Applying the linear form $\Lambda_{\Pi_F}$ constructed in \Cref{sec back to hecke modules}, we see that $\Lambda_{\Pi_F}\left(\Phi_c\left(\Xi_c(\delta)\right)\right)$ is given by a finite $\mathbf{Z}[1/p]$-linear combination of elements of the form
         $$\begin{dcases}
            \Lambda_{\Pi_F}\left(\ch(\mathbf{t}_a G(\mathcal{O}_F))\right),\ &\mathrm{if}\ b=0\\
            \Lambda_{\Pi_F}\left((p-1)p^{b-1}\ch(\mathbf{t}_a\mathbf{n}_b G(\mathcal{O}_F))\right),\ &\mathrm{if}\ b>0.
        \end{dcases}$$
        By \Cref{prop explicit Lambda}, each such term is also given by evaluating the spherical Hecke eigensystem $\Theta_{\Pi_F}$ of $\Pi_F$ at a specific Hecke operator independent of $\Pi_F$. Looking at the explicit expression of this Hecke operator in \Cref{prop explicit Lambda} and the values of the constants $\epsilon_n(b)$, we see that $\Lambda_{\Pi_F}\left(\Phi_c\left(\Xi_c(\delta)\right)\right)$ is given by evaluating $\Theta_{\Pi_F}$ on an integral Hecke operator, say $\mathcal{Q}\in \mathcal{H}_{G(F)}^\circ(\mathbf{Z}[1/p])$ (again independently of $\Pi_F$). On the other hand we have
    \begin{align}\label{eq: 24}
\Theta_{\Pi_F}(\mathcal{Q})&=\Lambda_{\Pi_F}\left(\Phi_c\left(\Xi_c(\delta)\right)\right)\\
\nonumber&=\Lambda_{\Pi_F}\left(\mathcal{P}_\delta\cdot\ch(G(\mathcal{O}_F))\right)\\
\nonumber&=\Lambda_{\Pi_F}\left(\mathcal{P}_\delta^{'}\right)\\
\nonumber&=\Theta_{\Pi_F}\left(\mathcal{P}_\delta^{'}(1-\mathcal{S}_F)\right).
    \end{align}
    where the second equality follows from \textcolor{Black}{the definition of $\mathcal{P}_\delta$} and Hecke equivariance of $\Xi_c$ and $\Phi_c$, the third equality follows from \textcolor{Black}{the} direct computation \textcolor{Black}{$$(\mathcal{P}_\delta\cdot\ch(G(\mathcal{O}_F)))(x)=\int_{G(F)}\mathcal{P}_\delta(g)\ch(G(\mathcal{O}_F))(xg)\ dg=\int_{G(F)}\ch(G(\mathcal{O}_F))(g)\mathcal{P}_{\delta}^{'}(g^{-1}x)\ dg=\mathcal{P}_\delta^{'}(x)$$}and the last equality follows from \eqref{eq: map Lambda} \textcolor{Black}{and \Cref{lem non-zero Asai periods}}. By the usual density argument for the family of $\Pi_F$ in the spectrum of the Hecke algebra, we conclude that $\mathcal{P}_\delta^{'}(1-\mathcal{S}_F)\in\mathcal{H}_{G(F)}^\circ(\mathbf{Z}[1/p])$ and thus so is $\mathcal{P}_\delta$. This concludes the proof of the first part. For the other two parts, as in \cite[Theorem $5.1.1$]{groutides2024rankinselbergintegralstructureseuler}, there is a commutative diagram (with no compact support condition on the right-hand side)
        \[\begin{tikzcd}
	{\mathcal{I}(G(F)/G(\mathcal{O}_F),\mathbf{Z}[1/p])} & {C^\infty(P(\mathbf{Q}_p)\backslash G(F)/ G(\mathcal{O}_F),\mathbf{Z}[1/p])} \\
	{\mathcal{I}(G(F)/G(\mathcal{O}_F)[p],\mathbf{Z}[1/p])} & {C^\infty(P(\mathbf{Q}_p)\backslash G(F)/ G(\mathcal{O}_F)[p],\mathbf{Z}[1/p])}
	\arrow["\Xi", from=1-1, to=1-2]
	\arrow["{\mathrm{Tr}}", from=2-1, to=1-1]
	\arrow["{\Xi[p]}", from=2-1, to=2-2]
	\arrow["{\mathrm{Tr}}", from=2-2, to=1-2]
\end{tikzcd}\]
Let $\delta\in\mathcal{I}(G(F)/G(\mathcal{O}_F)[p],\mathbf{Z}[1/p])$. The analogous claim to make in the Asai case, is that $\Xi(\mathrm{Tr}(\delta))$ is valued in $(p-1)\mathbf{Z}[1/p]$ on $Z(F)$. From the commutativity of the diagram above and \eqref{eq: 17}, it is enough to show that the image under the trace map of a function $\xi:=\ch(P(\mathbf{Q}_p)\mathbf{t}_a G(\mathcal{O}_F)[p])$ is valued in $(p-1)\mathbf{Z}[1/p]$. This image is given by
\begin{align}
    \sum_{\gamma\in G(\mathcal{O}_F)/G(\mathcal{O}_F)[p]}\xi((-)\gamma).
\end{align}
It is clear that a complete set of distinct coset representatives of $G(\mathcal{O}_F)/G(\mathcal{O}_F)[p]$ is given by
$$\scalemath{0.9}{
    \left[\begin{matrix}
        (\mathbf{Z}/p\mathbf{Z})^\times & \\
        & 1
    \end{matrix}\right]\ ,\ \left[\begin{matrix}
        (\mathbf{Z}/p\mathbf{Z})^\times & \\
        & 1
    \end{matrix}\right]\left[\begin{matrix}
        \alpha & \\
        & 1
    \end{matrix}\right]\ ,\ \left[\begin{matrix}
        (\mathbf{Z}/p\mathbf{Z})^\times & \\
        & 1
    \end{matrix}\right]\left[\begin{matrix}
        \alpha+(\mathbf{Z}/p\mathbf{Z})^\times & \\
        & 1
    \end{matrix}\right].}
$$
Since for every $u\in(\mathbf{Z}/p\mathbf{Z})^\times$ we have $$P(\mathbf{Q}_p)\mathbf{t}_a G(\mathcal{O}_F)[p]\left[\begin{smallmatrix}
    u & \\
    &1
\end{smallmatrix}\right]=P(\mathbf{Q}_p)\mathbf{t}_a G(\mathcal{O}_F)[p]$$
the claim follows. By looking at the explicit construction of the map $\Xi_c$ found in \cite[\S $3.3$]{groutides2024rankinselbergintegralstructureseuler} which \textcolor{Black}{was also recalled in the proof of} \Cref{prop Hecke equiv map}, we also see that $\Xi_c(\mathrm{Tr}(\delta))$ is valued in $(p-1)\mathbf{Z}[1/p]$ on $Z(F)$. Also, as in \textit{op.cit}, if $\delta$ is as in the second part of the theorem, then $\Xi_c(\mathrm{Tr}(\delta))=(1-\mathcal{S}_F)\cdot\Xi(\mathrm{Tr}(\delta))$ and $\Xi(\mathrm{Tr}(\delta))$ is already compactly supported thus we can directly apply $\Lambda_{\Pi_F}\circ \Phi_c$ to it. Combining this with \Cref{prop explicit Lambda} and the arguments in the proof of the first part of the theorem, the last two parts follow.
    \end{proof}
\end{thm}

\subsection{Local periods and integral values}
Using the Hecke-module results of the previous section, we are able to determine how local Asai periods interact with these types of integral structures. Or in other words, when they are evaluated on \textit{arbitrary integral test data} in the sense of \cite[\S $6.4$]{loeffler2021zetaintegralsunramifiedrepresentationsgsp4}. This would not have been possible to obtain, for such general input data, via unfolding and trying to compute the integrals involved directly.
\begin{thm}\label{prop integral period 1}
    Let $\Pi_F$ be an unramified Whittaker type $G(F)$-representation. Then, 
    \begin{align}\label{eq: 26}
        \mathrm{dim}\ \mathrm{Hom}_{G(\mathbf{Q}_p)}\left(\mathcal{S}(\mathbf{Q}_p^2)\otimes \Pi_F,\mathbf{1}\right)=1
    \end{align}
    with the period $\mathcal{Z}$ as a basis.
    Furthermore, suppose that the spherical Hecke eigensystem of $\Pi_F$ restricts to $\mathcal{H}_{G(F)}^\circ(\mathcal{R})\rightarrow\mathcal{R}$ for some $\mathbf{Z}[1/p]$-algebra $\mathcal{R}$, then:
    \begin{enumerate}
    \item For any $g\in G(F)$ and every $\phi\in\mathcal{S}(\mathbf{Q}_p^2)$ valued in $\vol_{G(\mathbf{Q}_p)}(\mathrm{Stab}_{G(\mathbf{Q}_p)}(\phi)\cap gG(\mathcal{O}_F) g^{-1})^{-1}\cdot\mathcal{R}$ we have $\mathcal{Z}\left(\phi\otimes gW_{\Pi_F}^\mathrm{sph}\right)\in \mathcal{R}$.
        \item For any $g\in G(F)$ and every $\phi\in\mathcal{S}_0(\mathbf{Q}_p^2)$ valued in $\vol_{G(\mathbf{Q}_p)}(\mathrm{Stab}_{G(\mathbf{Q}_p)}(\phi)\cap gG(\mathcal{O}_F)[p] g^{-1})^{-1}\cdot\mathcal{R}$ we have $$\mathcal{Z}\left(\phi\otimes gW_{\Pi_F}^\mathrm{sph}\right)\in \left\langle (p-1)L(\omega_{\Pi_F},0)^{-1},L^\mathrm{As}(\Pi_F,0)^{-1}\right\rangle\subseteq\mathcal{R}.$$
        \item For every $g\in G(F)$ and every $\phi\in\mathcal{S}(\mathbf{Q}_p^2)$ valued in $\vol_{G(\mathbf{Q}_p)}(\mathrm{Stab}_{G(\mathbf{Q}_p)}(\phi)\cap gG(\mathcal{O}_F)[p] g^{-1})^{-1}\cdot\mathcal{R}$ we have $$\mathcal{Z}\left(\phi\otimes gW_{\Pi_F}^\mathrm{sph}\right\rangle\in \left\langle p-1,L(\Pi_F,0)^{-1}\right\rangle\subseteq\mathcal{R}.$$
    \end{enumerate}
   
\end{thm}
 Before sketching the proof, we want to mention that the bound on the dimension in \eqref{eq: 26} follows generically (i.e. for almost all $\Pi_F)$ by \cite[Theorem $1$] {8b685a57-0995-30f2-97f8-c2f88420d820} where Kable proves a certain generic \textcolor{Black}{upper} bound in greater generality, which specializes to \eqref{eq: 26} in our case. The techniques in \textit{op.cit} are vastly different. Here we give an alternative proof of this result in our specific setup, which allows one to remove the generic condition and include all such $\Pi_F$.

 \begin{proof}
        The fact that the linear form $\mathcal{Z}$ is a non-zero element of the Hom-space in \eqref{eq: 26} follows from \Cref{lem non-zero Asai periods}(1). To bound the dimension, one argues in the same manner as \cite[Theorem $3.2.3$]{groutides2024rankinselbergintegralstructureseuler}, this time using \Cref{cor cyclicity of I}. For the two statements regarding integrality, one argues as in \cite[Theorem $6.2.1$]{groutides2024rankinselbergintegralstructureseuler}, i.e., we combine \cite[Proposition $4.4.1$]{Loeffler_2021} with \Cref{thm integral local factors} and \Cref{lem Asai Euler factor}.
    \end{proof}

\begin{rem}
    The assumption on the spherical Hecke eigensystem is a natural one to make from a global viewpoint, as we'll see later when we apply this to obtain our $\ell$-adic integrality results for periods attached to Hilbert modular forms.
\end{rem}

We also want to outline an interesting connection with distinguished poles of the Asai $L$-factor $L^\mathrm{As}(\Pi_F,s)$ (in the sense of \cite[Definition $3.1$]{matringe2008distinguishedrepresentationsexceptionalpoles}) and unramified Whittaker type $G(F)$-representations distinguished by $G(\mathbf{Q}_p)$. This was first studied in \cite{flicker1994quaternionic} where (specializing to our setup) the authors prove that every unramified Whittaker-type $\Pi_F$ of $G(F)$, with trivial central-character, is $G(\mathbf{Q}_p)$-distinguished. The proof of the above proposition directly constructs a non-zero $G(\mathbf{Q}_p)$-invariant linear form on $\Pi_F$ and actually shows that it is non-vanishing on the spherical vector of $\Pi_F$. In fact, by \cite{Flicker+1991+139+172} this linear form is unique up to non-zero scalar multiple (the proof in \textit{op.cit} deals with the irreducible case however the result extends to Whittaker type). By \cite[Theorem $3.1$]{matringe2008distinguishedrepresentationsexceptionalpoles} (which uses the functional equation due to \cite{flicker1993zeroes}), this implies that $L^\mathrm{As}(\Pi_F,s)$ has an \textit{exceptional pole} at $s=0$ and thus the linear form $\mathcal{Z}$ of \Cref{prop integral period 1} vanishes on $\mathcal{S}_0(\mathbf{Q}_p^2)\otimes \Pi_F$ which is in line with part $(2)$ of \Cref{prop integral period 1} as in this case both generators of the ideal present are identically zero. In addition, the normalized period $\Xi$ in $\mathrm{Hom}_{G(\mathbf{Q}_p)}(\Pi_F,\mathbf{1})$ which sends $W_{\Pi_F}^\mathrm{sph}$ to $1$, satisfies $\mathcal{Z}(\phi\otimes gW_{\Pi_F}^\mathrm{sph})=\Xi(gW_{\Pi_F}^\mathrm{sph})\phi(0)$. Thus, using this and \Cref{prop integral period 1}($3$), one can obtain the following integrality result for the period $\Xi$, constructed in \cite{flicker1994quaternionic}:
\begin{prop}
    Let $\Pi_F$ be an unramified Whittaker type $G(F)$-representation with trivial central character. Assume that its spherical Hecke eigensystem restricts to $\mathcal{H}_{G(F)}^\circ(\mathcal{R})\rightarrow\mathcal{R}$ for some $\mathbf{Z}[1/p]$-algebra $\mathcal{R}$. Then $\Xi$, the unique normalized $G(\mathbf{Q}_p)$-invariant period on $\Pi_F$, constructed in \cite{flicker1994quaternionic}, satisfies
    \begin{align*}
        \Xi\left(\vol_{G(\mathbf{
        Q}_p)}(G(\mathbf{Z}_p)\cap g G(\mathcal{O}_F)[p] g^{-1})^{-1}\cdot gW_{\Pi_F}^\mathrm{sph}\right)\in \left\langle p-1,L(\Pi_F,0)^{-1}\right\rangle\subseteq\mathcal{R}
    \end{align*}
    for all $g\in G(F)$.
\end{prop}
 
 Finally, we briefly discuss the corresponding situation of \cite[\S $6.3$]{groutides2024rankinselbergintegralstructureseuler}. The space of invariant periods $\mathrm{Hom}_{G(\mathbf{Q}_p)}(\Pi_F\otimes \pi_p,\mathbf{1})$ where $\Pi_F,\pi_p$ are irreducible admissible representations of $G(F)$ and $G(\mathbf{Q}_p)$, respectively, was originally studied in detail in \cite{prasad1992invariant} where one of the main theorems \textcolor{Black}{states} that for irreducible $\Pi_F$ and $\pi_p$ the Hom-space above is at most one dimensional (even for ramified quadratic extensions $F/\mathbf{Q}_p$). Prasads' proof uses the Gelfand-Kazhdan approach of invariant distributions. In \cite{loeffler2021gross} using different approaches, Loeffler extends this result to allow for one or both of $\Pi_F$ and $\pi_p$ to be reducible and of Whittaker type. The proof is not tailored to \textit{unramified} Whittaker type representations, and thus even within this family, the results in both \cite{prasad1992invariant} and \cite{loeffler2021gross}, distinguish between several cases in order to obtain uniqueness of such an invariant period. Using \Cref{thm cyclicity of M}, one can obtain a unified direct proof for the uniqueness of invariant periods for the family of unramified Whittaker type representations, together with abstract existence of good test vectors, in the $F=\mathbf{Q}_{p^2}$ setting. Additionally, we can obtain results on the \textit{integral} behavior of such a non-zero period using the Godement–Siegel projection map $\mathcal{S}(\mathbf{Q}_p^2)\rightarrow \pi_p$, in the same style as \cite[\S $6.3$]{groutides2024rankinselbergintegralstructureseuler}, this time using \Cref{prop integral period 1}.
 \section{Periods associated to Hilbert modular forms}\label{sec periods attached to HMF}
 Throughout this section, we fix a totally real quadratic field $E$, and prime $\ell$. We'll work with the algebraic groups $G=\GL_2$ and $\mathscr{G}=\mathrm{Res}_{E/\mathbf{Q}}G$. \textcolor{Black}{Additionally, we will only work with primes $p$ which are unramified in $E$.} It is clear that 
$$\scalemath{0.9}{\mathscr{G}(\mathbf{Q}_p)=\begin{dcases}
    G(E_{v_1})\times G(E_{v_2})=G(\mathbf{Q}_p)\times G(\mathbf{Q}_p),\ &\mathrm{if}\  p=v_1v_2\ \mathrm{split}\\
    G(E_v)=G(\mathbf{Q}_{p^2}),\ &\mathrm{if}\  p=v\ \mathrm{inert.}
\end{dcases}}$$
We have the standard maximal open compact subgroup in $\mathscr{G}(\mathbf{Q}_p)$, given by
$$\scalemath{0.9}{\mathscr{G}(\mathbf{Z}_p):=\begin{dcases}
    G(\mathcal{O}_{E_{v_1}})\times G(\mathcal{O}_{E_{v_2}})=G(\mathbf{Z}_p)\times G(\mathbf{Z}_p),\ &\mathrm{if}\  p=v_1v_2\ \mathrm{split}\\
    G(\mathcal{O}_{E_v})=G(\mathbf{Z}_{p^2}),\ &\mathrm{if}\  p=v\ \mathrm{inert}
\end{dcases}}$$
and the open compact determinant level subgroup of $\mathscr{G}(\mathbf{Z}_p)$, given by
\begin{align*}\label{eq: det level subgp}
   \scalemath{0.9}{ \mathscr{G}(\mathbf{Z}_p)[p]:=\begin{dcases}
    \{(g_1,g_2)\in \mathscr{G}(\mathbf{Z}_p)\ |\ \det(g_2)\in 1+p\mathbf{Z}_{p}\},\ &\mathrm{if}\ p\ \mathrm{split}\\
        \{g\in \mathscr{G}(\mathbf{Z}_p)\ |\ \det(g)\in 1+p\mathbf{Z}_{p^2}\},\ &\mathrm{if}\ p\ \mathrm{inert}.
    \end{dcases}}
\end{align*}

\begin{rem}
    Of course, for inert primes, this setup matches the local setup of \Cref{sec local results}. For split primes, it matches the local setup of \cite{groutides2024rankinselbergintegralstructureseuler}, and as in \textit{op.cit}, the results will be independent of the choice between $\det(g_1)$ and $\det(g_2)$ in the definition of $\mathscr{G}(\mathbf{Z}_p)[p]$.
\end{rem}
As usual, given an unramified prime $p$, in $E$, we set $\mathcal{H}_{\mathscr{G}(\mathbf{Q}_p)}^\circ:=C_c^\infty(\mathscr{G}(\mathbf{Z}_p)\backslash\mathscr{G}(\mathbf{Q}_p)/\mathscr{G}(\mathbf{Z}_p))$; the spherical Hecke algebra of $\mathscr{G}(\mathbf{Q}_p)$ under the standard convolution with respect to the Haar measure on $\mathscr{G}(\mathbf{Q}_p)$ which gives $\mathscr{G}(\mathbf{Z}_p)$ volume $1$.
 \begin{defn}
 Let $\underline{k}=(k_1,k_2)\in\mathbf{Z}_{\geq 2}$ and $\underline{t}=(t_1,t_2)\in\mathbf{Z}_{\geq 0}$, with $w:=k_1+2t_1=k_2+2t_2$. A quadratic Hilbert cuspidal eigenform $\mathbf{f}$ associated to $E$, of level $\mathfrak{n}\subseteq \mathcal{O}_E$ with $\ell\nmid\mathfrak{n}$, and weight $(\underline{k},\underline{t})$, is a normalized eigenform in $S_{(\underline{k},\underline{t})}(U_1(\mathfrak{n}),\mathbf{C})$, defined as in \cite[\S $4.1$]{lei2018euler}.
 \end{defn}
\subsection{Associated representations}
Let $\mathbf{f}$ be a quadratic Hilbert cuspidal eigenform as in the above definition. We write $\Pi_\mathbf{f}$ for the associated irreducible, automorphic representation of $G(\mathbf{A}_E)$. The representation $\Pi_\mathbf{f}$ factorizes into a restricted tensor product of irreducible admissible representations of $G(E_v)$ where $v$ runs through places of $E$. However, as in \cite{8b685a57-0995-30f2-97f8-c2f88420d820}, we want to consider such a factorization as being indexed over places of $\mathbf{Q}$. In other words, we regard it as an automorphic representation of $\mathscr{G}(\mathbf{A})$. Thus, if $p$ is a place of $\mathbf{Q}$ which splits in $E$ as $v_1v_2$, we write $\Pi_{\mathbf{f},p}$ for $\Pi_{\mathbf{f},v_1}\boxtimes \Pi_{\mathbf{f},v_2}$ regarded as a representation of $G(E_{v_1})\times G(E_{v_2})= G(\mathbf{Q}_p)\times G(\mathbf{Q}_p)$. Alternatively,
if $p$ is non-split, then $\Pi_{\mathbf{f},p}$ denotes $\Pi_{\mathbf{f},v}$, where $v$ is the unique place of $E$ above $p$, regarded as a representation of $G(E_v)=G(\mathbf{Q}_{p^2})$. Using this notation, the representation $\Pi_\mathbf{f}$ decomposes as
$\Pi_\mathbf{f}\simeq \bigotimes_{p}^{'}\Pi_{\mathbf{f},p}.$

  For an ideal $\mathfrak{a}$ of $\mathcal{O}_E$, we write $\lambda_\mathfrak{a}(\mathbf{f})$ for the $\mathfrak{a}$-th Hecke eigenvalue of $\mathbf{f}$ as in \cite[\S $4.1$]{lei2018euler}. For varying $\mathfrak{a}$, the $\lambda_\mathfrak{a}(\mathbf{f})$ are algebraic integers and generate a number field, which we denote by $L_\mathbf{f}$, containing the central character of $\Pi_\mathbf{f}$, which we denote by $\omega_{\Pi_\mathbf{f}}$. \textcolor{Black}{We explicitly have $\omega_{\Pi_\mathbf{f}}=|\cdot|_{\mathbf{A}_E}^{2-w}\epsilon_\mathbf{f}$, where $\epsilon_\mathbf{f}$ is the adelization of the Nebentype of $\mathbf{f}$.} We fix an arbitrary place $v$, of $L_\mathbf{f}$, lying above our fixed prime $\ell$. We write $\mathbf{L}_\mathbf{f}:=(L_\mathbf{f})_v$; a finite extension of $\mathbf{Q}_\ell$.

We let $S$ be the set of places of $\mathbf{Q}$ containing: $2,\ell,\infty$ and primes $p|\mathrm{Nm}_{E/\mathbf{Q}}(\mathfrak{n})\Delta_E$\textcolor{Black}{, where $\Delta_E$ denotes the discriminant of $E$}. Then, for every $p\notin  S$, the $\mathscr{G}(\mathbf{Q}_p)$-representation $\Pi_{\mathbf{f},p}$ is an unramified principal-series \textcolor{Black}{of $G(\mathbf{Q}_{p^2})$ if $p$ is inert, and a product of two unramified principal series of $G(\mathbf{Q}_p)$ if $p$ is split}. Following the same indexing, if a finite place $p\notin S$, splits as $p=v_1v_2$ we write $\omega_{\Pi_\mathbf{f}}(\varpi_p)$ for $\omega_{\Pi_\mathbf{f}}(\varpi_{v_1})\omega_{\Pi_\mathbf{f}}(\varpi_{v_2})$, where $\varpi_{v_1},\varpi_{v_2}$ are uniformizers of $E_{v_1},E_{v_2}$ respectively. On the other hand, if it's inert, we write $\omega_{\Pi_\mathbf{f}}(\varpi_p)$ for $\omega_{\Pi_\mathbf{f}}(\varpi_v)$, where $\varpi_v$ is a uniformizer of $E_v$ and $p|v$ (we adopt the same convention for $\epsilon_\mathbf{f})$. For $p\notin S$, we also write $L_p(\omega_{\Pi_\mathbf{f}},s):=\frac{1}{1-\omega_{\Pi_\mathbf{f}}(\varpi_p)p^s}$ and $$L^\mathrm{As}_p(\Pi_\mathbf{f},s):=\begin{dcases}
    L^\mathrm{As}(\Pi_{\mathbf{f},p} , s),\ &\mathrm{if}\ p\ \mathrm{is}\ \mathrm{inert}\\
    L(\Pi_{\mathbf{f},p},s),\ &\mathrm{if}\ p\ 
    \mathrm{splits}.
\end{dcases}$$
For $p$ inert $L^\mathrm{As}(\Pi_{\mathbf{f},p} , s)$ denotes the local Asai $L$-factor of \Cref{lem Asai Euler factor} and for $p=v_1v_2$ split, $L(\Pi_{\mathbf{f},p},s)=L(\Pi_{\mathbf{f},v_1}\boxtimes\Pi_{\mathbf{f},v_2},s)$ denotes the local Rankin-Selberg convolution $L$-factor, \textcolor{Black}{appeared} in \cite{groutides2024rankinselbergintegralstructureseuler}. 

We write $\rho_\mathbf{f}^\mathrm{As}=\rho_{\mathbf{f},v}^\mathrm{As}$, for the $4$-dimensional $\mathbf{L}_\mathbf{f}$-linear Asai Galois representation of $\mathrm{Gal}(\overline{\mathbf{Q}}/\mathbf{Q})$ attached to $\mathbf{f}$, which is given by the tensor induction $\otimes$-$\mathrm{Ind}^\mathbf{Q}_E(\rho_\mathbf{f}^\mathrm{std})\otimes \mathbf{L}_\mathbf{f}(t_1+t_2)$. The cyclotomic twist by $t_1+t_2$ is the natural one to consider as pointed out in \cite[Remark $4.3.3$]{lei2018euler}. Here $\rho_\mathbf{f}^\mathrm{std}$ denotes the ``standard'' $2$-dimensional Galois representation of $\mathrm{Gal}(\overline{\mathbf{Q}}/E)$ attached to $\mathbf{f}$, due to Blasius–Rogawski–Taylor. It is known that the Asai representation is unramified outside $S$. As usual, for a prime $p\notin S$, we write $P_p^\mathrm{As}(\mathbf{f},X)$ for the Artin local factor given by $\det(1-X\mathrm{Frob}_p^{-1}|\rho_\mathbf{f}^\mathrm{As})$ where $\mathrm{Frob}_p$ is the arithmetic Frobenius at $p$. Finally, for such prime $p$, we set $L_p^\mathrm{As}(\mathbf{f},s):=\frac{1}{P_p^\mathrm{As}(\mathbf{f},p^{-s})}$.

\subsection{The period}
This section adapts the result of \cite[\S $6.4$]{groutides2024rankinselbergintegralstructureseuler} from the Rankin-Selberg case to the Asai case. That is, from the convolution of two classical modular forms, to Hilbert modular forms. Here we present the result directly at the level of Asai periods on the unramified part of $\Pi_\mathbf{f}$ without realizing them as a certain limit of ratios of global and local zeta integrals, and $L$-functions. This can also be done if one wishes to, using \cite[\S $4$]{8b685a57-0995-30f2-97f8-c2f88420d820}.

\subsubsection{Normalization}
We regard $\Pi_\mathbf{f}^S$ (the $S$-finite part of $\Pi_\mathbf{f}$) as a representation of $\mathscr{G}(\mathbf{A}^S)$, and we consider the space of Asai periods
\begin{align}
    \mathrm{Hom}_{G(\mathbf{A}^S)}\left(\mathcal{S}((\mathbf{A}^S)^2)\otimes \Pi_\mathbf{f}^S,\mathbf{1}\right).
\end{align}
We actually know from the local tools presented in \Cref{sec local results} and the corresponding ones in \cite[\S $6$]{groutides2024rankinselbergintegralstructureseuler}, that this space is one dimensional. If $p=v_1v_2$ is split, then we fix additive characters $\psi_{v_i}:E_{v_i}\rightarrow\mathbf{C}^\times$ of conductor one such that $\psi_{v_2}=\psi_{v_1}^{-1}$, as characters of $\mathbf{Q}_p$, under the canonical isomorphisms $\mathbf{Q}_p\simeq E_{v_i}$. If $p=v$ is inert, then we fix an additive character $\psi_v:\mathbf{Q}_{p^2}\simeq E_v\rightarrow\mathbf{C}^\times$ as in \Cref{sec zeta integrals}. We can identify $\Pi_\mathbf{f}^S=\otimes_{p\notin S}^{'}\mathcal{W}(\Pi_{\mathbf{f},p})$, where $$\mathcal{W}(\Pi_{\mathbf{f},p}):=\begin{dcases}\mathcal{W}(\Pi_{\mathbf{f},v_1},\psi_{v_1})\boxtimes \mathcal{W}(\Pi_{\mathbf{f},v_2},\psi_{v_2}),\ &\mathrm{if}\ p=v_1v_2\\
\mathcal{W}(\Pi_{\mathbf{f},v},\psi_v),\ &\mathrm{if}\ p=v.
\end{dcases}
$$Under this identification, we have a canonical normalized spherical vector $W_{\Pi_\mathbf{f}^S}^\mathrm{sph}\in\Pi_\mathbf{f}^S$, that factorizes into the standard normalized spherical vectors in each local Whittaker model which map the identity element to $1$. The unique non-zero period in $ \mathrm{Hom}_{G(\mathbf{A}^S)}\left(\mathcal{S}((\mathbf{A}^S)^2)\otimes \Pi_\mathbf{f}^S,\mathbf{1}\right)$ is normalized to map the unramified vector $\ch((\hat{\mathbf{Z}}^S)^2)\otimes W_{\Pi_\mathbf{f}^S}^\mathrm{sph}$ to $1$.

\begin{thm}\label{thm hilbert period}
Let $\mathcal{Z}_\mathbf{f}$ be the unique, normalized, non-zero period in $\mathrm{Hom}_{G(\mathbf{A}^S)}\left(\mathcal{S}((\mathbf{A}^S)^2)\otimes \Pi_\mathbf{f}^S,\mathbf{1}\right)$.Then the following are true:
\begin{enumerate}
    \item For any $g\in\mathscr{G}(\mathbf{A}^S)$ and any decomposable Schwartz function $\Phi=\otimes_{p\notin S} \Phi_p\in \mathcal{S}((\mathbf{A}^S)^2)$ where each $\Phi_p$ is valued in $\vol_{G(\mathbf{Q}_p)}(\mathrm{Stab}_{G(\mathbf{Q}_p)}(\Phi_p)\cap g_p\mathscr{G}(\mathbf{Z}_p)g_p^{-1})^{-1}\cdot\mathcal{O}_{L_\mathbf{f}}$, the period $\mathcal{Z}_\mathbf{f}$ satisfies 
    $$\mathcal{Z}_\mathbf{f}(\Phi\otimes gW_{\Pi_\mathbf{f}}^\mathrm{sph})\in\mathcal{O}_{\mathbf{L}_\mathbf{f}}.$$
    \item Suppose, moreover, that $S_0$ denotes a finite set of primes disjoint from $S$, and for each $p\in S_0$, $\Phi_p$ is valued in $\vol_{G(\mathbf{Q}_p)}(\mathrm{Stab}_{G(\mathbf{Q}_p)}(\Phi_p)\cap g_p\mathscr{G}(\mathbf{Z}_p)[p]g_p^{-1})^{-1}\cdot\mathcal{O}_{L_\mathbf{f}}$. Then the period $\mathcal{Z}_\mathbf{f}$ also satisfies
    \begin{align}\label{eq: 27}\mathcal{Z}_\mathbf{f}\left(\Phi\otimes gW_{\Pi_\mathbf{f}^S}^\mathrm{sph}\right)\cdot \left(\prod_{\substack{p\in S_0\\ \Phi_p\notin \mathcal{S}_0(\mathbf{Q}_p^2)}} L_p(\omega_{\Pi_\mathbf{f}},0)^{-1}\right)\in\prod_{p\in S_0}\left\langle p-1, L_p^\mathrm{As}(\mathbf{f},1-(t_1+t_2))^{-1}\right\rangle\subseteq \mathcal{O}_{\mathbf{L}_\mathbf{f}}.\end{align}
\end{enumerate}
    In particular, if \textcolor{Black}{$h_E^+$ denotes the narrow class number of $E$, $(h_E^+\cdot \#(\mathcal{O}_E/\mathfrak{n})^\times ,\ell)=1$, and $\epsilon_\mathbf{f}(\varpi_p)\neq 1$ for all $p\in S_0\cap \{p:\Phi_p\notin \mathcal{S}_0(\mathbf{Q}_p^2)\}$}, then the containment in \eqref{eq: 27} holds without the bracketed product of Tate $L$-factors on the left. \textcolor{Black}{The same is true if $S_0\cap \{p:\Phi_p\notin \mathcal{S}_0(\mathbf{Q}_p^2)\}$ is empty.}
\end{thm}
\begin{proof}
    The period $\mathcal{Z}_\mathbf{f}$ can be written down explicitly as the product, over all $p\notin S$, of normalized local periods $\mathcal{Z}_p$ on $\mathcal{S}(\mathbf{Q}_p^2)\otimes\Pi_{\mathbf{f},p}$. For inert primes, this local period is the one of \Cref{lem non-zero Asai periods}(1), and for split primes it is the one of \cite[Definition $6.4.4$]{groutides2024rankinselbergintegralstructureseuler}. Then 
    $$\mathcal{Z}_\mathbf{f}\left(\Phi\otimes gW_{\Pi_\mathbf{f}^S}^\mathrm{sph}\right)=\left\{\prod_{\substack{p\notin S\\ p=v_1v_2}}\mathcal{Z}_p\left(\Phi_p\otimes g_p\left(W_{\Pi_{\mathbf{f},v_1}}^\mathrm{sph}\otimes W_{\Pi_{\mathbf{f},v_2}}^\mathrm{sph}\right) \right)\right\}\times \left\{\prod_{\substack{p\notin S\\ p=v}}\mathcal{Z}_p\left(\Phi_p\otimes g_pW_{\Pi_{\mathbf{f},v}}^\mathrm{sph}\right)\right\}$$
    where for almost all primes $p$, each local piece is equal to $1$.
    Let $p\notin S$, and let $u$ be a place of $E$ above $p$. By \cite[ \S $7.2.4$]{grossi2020norm}, the Satake parameters of the unramified $G(E_u)$-principal-series $\Pi_{\mathbf{f},u}$, are given by $q_u^{-1/2}\alpha_u,q_u^{-1/2}\beta_u$, where $\alpha_u,\beta_u$ are the roots of the polynomial $X^2-\lambda_u(\mathbf{f})X+q_u^{w-1}\omega_{\Pi_\mathbf{f}}(\varpi_u)$ and $q_u$ denotes the size of the residue field of $E_u$. Thus, the spherical Hecke eigenvalues of $\Pi_{\mathbf{f},u}$ are given by $\lambda_u(\mathbf{f})$ and $q_u^{w-2}\omega_{\Pi_\mathbf{f}}(\varpi_u)$, which are both elements of $\mathcal{O}_{\mathbf{L}_\mathbf{f}}$. Hence, for every $p\notin S$, the spherical Hecke eigensystem for the $\mathscr{G}(\mathbf{Q}_p)$-representation $\Pi_{\mathbf{f},p}$ restricts to a morphism $\mathcal{H}_{\mathscr{G}(\mathbf{Q}_p)}^\circ(\mathcal{O}_{\mathbf{L}_\mathbf{f}})\rightarrow \mathcal{O}_{\mathbf{L}_\mathbf{f}}.$ The integrality and the containment in \eqref{eq: 27} then follow by an application of \Cref{prop integral period 1} for the local periods at inert primes, \cite[Theorem $6.2.1$]{groutides2024rankinselbergintegralstructureseuler} for the local periods at split primes, and \cite[Corollary $7.2.4$]{grossi2020norm} which gives that $L_p^\mathrm{As}(\Pi_\mathbf{f},s)=L_p^\mathrm{As}(\mathbf{f},1+s-(t_1+t_2))$ for all $p\notin S$. For the last part, recall that $\omega_{\Pi_\mathbf{f}}$ is given by $|\cdot|_{\mathbf{A}_E}^{2-w}\epsilon_\mathbf{f}$ and $\epsilon_\mathbf{f}$ is of finite order dividing $h_E^+\cdot \#(\mathcal{O}_E/\mathfrak{n})^\times$. \textcolor{Black}{Thus if the further assumptions on $h_E^+$ and $\epsilon_\mathbf{f}$ hold, the bracketed product of Tate $L$-factors can be removed from $\eqref{eq: 27}$: Indeed, $L_p(\omega_{\Pi_\mathbf{f}},0)^{-1}=(1-\epsilon_\mathbf{f}(\varpi_p))-(p^{2(w-2)}-1)\epsilon_\mathbf{f}(\varpi_p)$. Hence, from the coprimality condition and the assumption $\epsilon_\mathbf{f}(\varpi_p)\neq1$, the term $1-\epsilon_\mathbf{f}(\varpi_p)$ is a unit in $\mathcal{O}_{\mathbf{L}_\mathbf{f}}$. If $p-1\not\equiv0\mod\ell$, then it generates the unit ideal. Hence, by the proof of the first part or \Cref{prop integral period 1}(1), we may assume that $p\equiv 1\mod\ell$. In this case we have $v_\ell\left((p^{2(w-2)}-1)\epsilon_\mathbf{f}(\varpi_p)\right)>0$ and hence $L_p(\omega_{\Pi_\mathbf{f}},0)^{-1}$ is an $\ell$-adic unit so we are done.}
\end{proof}

 \section{Asai-Flach norm relations}
 In this final section, we study the integral behavior of the motivic Asai-Flach tame norm relations. Once again by adapting and extending our earlier work in \cite{groutides2024rankinselbergintegralstructureseuler}, and the approach in \cite{Loeffler_2021}, we show in \Cref{thm euler system}$(1)$, that the motivic Asai-Flach Euler system tame norm relations, are optimal in the strongest possible integral sense; i.e., any construction of this type, with any choice of integral input data in the sense of \cite{Loeffler_2021}, does not only yield motivic classes which lie in the integral \'etale realization (\Cref{Prop integral classes} \&  \cite[Proposition $9.5.2$]{Loeffler_2021}), but also local factors which are always integrally divisible by the Asai Euler factor $\mathcal{P}_{p,\mathrm{As}^*}^{'}(\mathrm{Frob}_p^{-1})$, modulo $\scalemath{0.9}{p-1}$, proving a conjecture of Loeffler in this setup. 

 Additionally, by carefully choosing one such integral collection, we are able to obtain the most general version of integral Asai-Flach tame norm relations; i.e. we obtain motivic integral classes which satisfy tame norm relations, for all unramified primes, in motivic cohomology, with the expected Euler factor $\mathcal{P}_{p,\mathrm{As}^*}^{'}(\mathrm{Frob}_p^{-1})$. Finally, as mentioned in the introduction, the proof also holds for non-trivial coefficient sheaves, and if set up properly, \Cref{thm euler system}$(2)$ recovers \cite[Theorem $7.3.2$]{grossi2020norm} after passing to Galois cohomology.

\subsection{Passing to a smaller group}\label{sec setup}
 Once again, let $E/\mathbf{Q}$ be a totally real quadratic field. For applications to the Asai-Flach Euler system, we generally want to work with certain slightly smaller groups which allow for cyclotomic base change. Following \cite{lei2018euler} (and consequently \cite{dimitrov2004galoisrepresentationsmodulop}) we keep the definition of the group $\mathscr{G}=\mathrm{Res}_{E/\mathbf{Q}}\GL_2$ from the previous section, and we also define the following algebraic groups:
 $$D:=\mathrm{Res}_{E/\mathbf{Q}}\GL_1,\ \ \mathscr{G}^*:=\mathscr{G}\times_D\GL_1.$$
 fibered over the determinant map. We have natural embeddings $G\hookrightarrow\mathscr{G}$ and $G\hookrightarrow \mathscr{G}^*, g\mapsto (g,\det(g))$. Finally, $\mathscr{G}^*$ itself admits a natural embedding into $\mathscr{G}$ via projection to the first factor. It is clear that
 $$\scalemath{0.9}{\mathscr{G}^*(\mathbf{Q}_p)=\begin{dcases}
    G(E_{v_1})\times_{\mathbf{Q}_p^\times} G(E_{v_2}),\ &\mathrm{if}\  p=v_1v_2\ \mathrm{split}\\
    \{g\in G(E_v)\ |\ \det(g)\in\mathbf{Q}_p^\times\},\ &\mathrm{if}\  p=v\ \mathrm{inert.}
\end{dcases}}$$
 We again have a natural maximal open compact subgroup $\mathscr{G}^*(\mathbf{Z}_p)\subseteq \mathscr{G}^*(\mathbf{Q}_p)$, given by 
 $$\scalemath{0.9}{\mathscr{G}^*(\mathbf{Z}_p):=\begin{dcases}
    G(\mathcal{O}_{E_{v_1}})\times_{\mathbf{Z}_p^\times} G(\mathcal{O}_{E_{v_2}}),\ &\mathrm{if}\  p=v_1v_2\ \mathrm{split}\\
    \{g\in G(\mathcal{O}_{E_v})\ |\ \det(g)\in\mathbf{Z}_p^\times\},\ &\mathrm{if}\  p=v\ \mathrm{inert.}
\end{dcases}}$$
 and a canonical determinant open compact level subgroup, given by $$\scalemath{0.9}{\mathscr{G}^*(\mathbf{Z}_p)[p]:=\{g\in\mathscr{G}^*(\mathbf{Z}_p)\ |\ \det(g)\in 1\textcolor{Black}{+p\mathbf{Z}_p}\}}
 $$ which makes perfect sense for this smaller group $\mathscr{G}^*$, and there is no ambiguity.
  As usual, we write $\mathcal{H}_{\mathscr{G}(\mathbf{Q}_p)}^\circ$, respectively  $\mathcal{H}_{\mathscr{G}^*(\mathbf{Q}_p)}^\circ$, for the spherical Hecke algebra of $\mathscr{G}(\mathbf{Q}_p)$, respectively $\mathscr{G}^*(\mathbf{Q}_p)$, with respect to $\mathscr{G}(\mathbf{Z}_p)$, respectively $\mathscr{G}^*(\mathbf{Z}_p)$. These are once again commutative, unital, convolution algebras where the Haar measures are normalized as usual to give $\mathscr{G}(\mathbf{Z}_p)$ and $\mathscr{G}^*(\mathbf{Z}_p)$ volume $1$, respectively. Once again, we will often canonically identify $\mathbf{Q}_p\simeq E_{v_1}\simeq E_{v_2}$ if $p=v_1v_2$ splits, and $\mathbf{Q}_{p^2}\simeq E_v=:E_p$ if $p=v$ is inert, in $E.$

 It is not hard to see that  for $p$ unramified in $E$, the group $\mathscr{G}^*(\mathbf{Q}_p)$ admits almost the same Cartan decomposition as the larger group $\mathscr{G}(\mathbf{Q}_p)$ but with $\mathscr{G}^*(\mathbf{Z}_p)$-double cosets instead of $\mathscr{G}(\mathbf{Z}_p)$-double cosets. In other words 
 $$\mathscr{G}^*(\mathbf{Q}_p)=\begin{dcases}\bigsqcup_{\substack{n_1\geq n_2\\ m_1\geq m_2\\ n_1+n_2=m_1+m_2}} \mathscr{G}^*(\mathbf{Z}_p) (t(n_1,n_2), t(m_1,m_2))\mathscr{G}^*(\mathbf{Z}_p),\ &\mathrm{if}\ p\ \mathrm{splits}\\
 \bigsqcup_{\substack{n_1\geq n_2}} \mathscr{G}^*(\mathbf{Z}_p)t(n_1,n_2) \mathscr{G}^*(\mathbf{Z}_p),\ &\mathrm{if}\ p\ \mathrm{is}\ \mathrm{inert}
 \end{dcases}$$
 where recall that $t(a,b)=\mathrm{diag}(p^a,p^b)$. Using this, it follows with a bit of work that for every unramified $p$, we have an injective Hecke algebra homomorphism
 $$\iota_p:\mathcal{H}_{\mathscr{G}^*(\mathbf{Q}_p)}^\circ\hookrightarrow \mathcal{H}_{\mathscr{G}(\mathbf{Q}_p)}^\circ\ ,\ \ \  \ch\left(\mathscr{G}^*(\mathbf{Z}_p)\ g\ \mathscr{G}^*(\mathbf{Z}_p)\right)\mapsto\ch(\mathscr{G}(\mathbf{Z}_p)\ g\  \mathscr{G}(\mathbf{Z}_p))$$
We set for each $a\in\mathbf{Z}_{\geq 0}$
\begin{align*}\mathcal{T}_*(p^a)&:=\begin{dcases}
    \ch(\mathscr{G}^*(\mathbf{Z}_p) (t(p^a,\textcolor{Black}{0}),t(p^a,\textcolor{Black}{0}))\mathscr{G}^*(\mathbf{Z}_p)),\ &\mathrm{if}\ p\ \mathrm{splits}\\
    \ch(\mathscr{G}^*(\mathbf{Z}_p) t(p^a,\textcolor{Black}{0})\mathscr{G}^*(\mathbf{Z}_p)),\ &\mathrm{if}\ p\ \mathrm{is}\ \mathrm{inert}.
\end{dcases}\\
\mathcal{S}_*(p)&:=\begin{dcases}
    \ch(\mathscr{G}^*(\mathbf{Z}_p) (t(p,p),t(p,p))\mathscr{G}^*(\mathbf{Z}_p)),\ &\mathrm{if}\ p\ \mathrm{splits}\\
    \ch(\mathscr{G}^*(\mathbf{Z}_p) t(p,p)\mathscr{G}^*(\mathbf{Z}_p)),\ &\mathrm{if}\ p\ \mathrm{is}\ \mathrm{inert}.
    \end{dcases}
\end{align*}
The Hecke operators $\mathcal{T}_*(p^a)$ and $\mathcal{S}_*(p)$ are clearly contained in $\mathcal{H}_{\mathscr{G}^*(\mathbf{Q}_p)}^\circ$. Additionally, in our previous notation and the notation of \cite[\S $2.1$]{groutides2024rankinselbergintegralstructureseuler}, we have $\iota_p(\mathcal{T}_*(p))=\mathcal{T}_{p,1}\mathcal{T}_{p,2}$, $\iota_p(\mathcal{S}_*(p))=\mathcal{S}_{p,1}\mathcal{S}_{p,2}$ if $p$ splits, and $\iota_p(\mathcal{T}_*(p))=\mathcal{T}_{E_p}, \iota_p(\mathcal{S}_*(p))=\mathcal{S}_{E_p}$ if $p$ is inert.
\begin{defn}
    We define the local Asai Euler factor $\mathcal{P}_{p,\mathrm{As}_*}(X)\in\mathcal{H}_{\mathscr{G}^*(\mathbf{Q}_p)}^\circ(\mathbf{Z}[1/p])[X]$ by 
    $$\scalemath{0.9}{\mathcal{P}_{p,\mathrm{As}_*}(X):=\begin{dcases}
        1-\tfrac{1}{p}\mathcal{T}_*(p)X+\left(\tfrac{1}{p^2}\mathcal{T}_*(p)^2-\tfrac{1}{p^2}\mathcal{T}_*(p^2)-\mathcal{S}_*(p)\right)X^2-\tfrac{1}{p}\mathcal{S}_*(p)\mathcal{T}_*(p)X^3+ \mathcal{S}_*(p)^2 X^4,\ &\mathrm{if}\ p\ \mathrm{splits}\\
        \left(1-\tfrac{1}{p}\mathcal{T}_*(p)X+\mathcal{S}_*(p)X^2\right)\left(1-\mathcal{S}_*(p) X^2\right),\ &\mathrm{if}\ p\ \mathrm{is}\ \mathrm{inert}.
    \end{dcases}}$$
\end{defn}
\begin{rem}
Using a simple computation in the spherical Hecke algebra of $G(\mathbf{Q}_p)$, it follows that for split $p$, we have $\iota_p\left(\mathcal{P}_{p,\mathrm{As}_*}\right)(X)=\mathcal{P}_p(X)$ where $\mathcal{P}_p(X)$ is as in \cite[\S $4.3$]{groutides2024rankinselbergintegralstructureseuler}. On the other hand, if $p$ is inert, then $\iota_p\left(\mathcal{P}_{p,\mathrm{As}_*}\right)(X)=\mathcal{P}_{E_p,\mathrm{As}}(X)$ where $\mathcal{P}_{E_p,\mathrm{As}}(X)$ is as in \Cref{lem Asai Euler factor}.
\end{rem}

It is important to note that
the cyclicity (and freeness) of the $\mathcal{H}_{\mathscr{G}^*(\mathbf{Q}_p)}^\circ$-modules $\mathcal{I}(\mathscr{G}^*(\mathbf{Q}_p)/\mathscr{G}^*(\mathbf{Z}_p))$ still holds for the smaller groups $\mathscr{G}^*(\mathbf{Q}_p)$ instead of $\mathscr{G}(\mathbf{Q}_p)$.  

\begin{prop}
   Let $p$ be unramified in $E$. The module $\mathcal{I}(\mathscr{G}^*(\mathbf{Q}_p)/\mathscr{G}^*(\mathbf{Z}_p))$ is cyclic and free over $\mathcal{H}_{\mathscr{G}^*(\mathbf{Q}_p)}^\circ$, generated by $\ch(\mathbf{Z}_p^2)\otimes\ch(\mathscr{G}^*(\mathbf{Z}_p))$, and there's a $\mathcal{H}_{\mathscr{G}^*(\mathbf{Q}_p)}^\circ$-equivariant morphism $$i_p:\mathcal{I}(\mathscr{G}^*(\mathbf{Q}_p)/\mathscr{G}^*(\mathbf{Z}_p))\hookrightarrow \mathcal{I}(\mathscr{G}(\mathbf{Q}_p)/\mathscr{G}(\mathbf{Z}_p)),\ \phi\otimes \ch(g \mathscr{G}^*(\mathbf{Z}_p)) \mapsto \phi\otimes \ch(g \mathscr{G}(\mathbf{Z}_p))$$
where the Hecke action on the right is via $\iota_p$.
    \begin{proof}
        The cyclicity proof is essentially identical to the proof of \cite[Theorem $3.1.1$]{groutides2024rankinselbergintegralstructureseuler} if $p$ splits, and \Cref{cor cyclicity of I} if $p$ is inert. The Hecke equivariance of the map $i_p$ (for both split and inert primes) basically follows from the following fact: If $F/\mathbf{Q}_p$ is a finite unramified extension (possibly $F=\mathbf{Q}_p)$, and $\lambda\in\mathbf{Z}_{\geq 0}$ then every coset representative appearing in the decomposition
        $$\scalemath{0.9}{G(\mathcal{O}_F)\left[\begin{smallmatrix}
            p^\lambda & \\
            & 1
        \end{smallmatrix}\right] G(\mathcal{O}_F)=\left(\bigsqcup_{\beta\in \mathcal{O}_F/p^{\lambda}\mathcal{O}_F}\left[\begin{smallmatrix}
            p^{\lambda} & \beta\\
            & 1
        \end{smallmatrix}\right] G(\mathcal{O}_F)\right)  \sqcup \left(\bigsqcup_{\substack{0<i<\lambda\\ \\ \beta\in(\mathcal{O}_F/p^{i}\mathcal{O}_F)^\times}}\left[\begin{smallmatrix}
            p^{i} & \beta\\
            &p^{\lambda-i}
        \end{smallmatrix}\right] G(\mathcal{O}_F)\right)\sqcup\left[\begin{smallmatrix}
            1 & \\
            &p^{\lambda}
        \end{smallmatrix}\right] G(\mathcal{O}_F)}$$
        can be written as $k_1\left[\begin{smallmatrix}
            p^\lambda & \\
            & 1
        \end{smallmatrix}\right]k_2$ with $k_1,k_2\in G(\mathcal{O}_F)$ and $\det(k_1)=\det(k_2)=1$.
        Freeness follows from freeness for the bigger groups together with the existence and properties of the maps $\iota_p$ and $i_p$.
    \end{proof}
\end{prop}
\begin{defn}\label{def freeness*}
  Consequently, for an element $\delta_p^*\in\mathcal{I}(\mathscr{G}^*(\mathbf{Q}_p)/\mathscr{G}^*(\mathbf{Z}_p))$, we write $\mathcal{P}_{\delta_p^*}$ for the unique Hecke operator in $\mathcal{H}_{\mathscr{G}^*(\mathbf{Q}_p)}^\circ$ which satisfies $\mathcal{P}_{\delta_p^*}\cdot (\ch(\mathbf{Z}_p^2)\otimes \ch(\mathscr{G}^*(\mathbf{Z}_p)))=\delta_p^*$.
  \end{defn}
\subsection{Integrality and local factors}
We will state the results in motivic cohomology with trivial coefficients for ease of notation and exposition. However, we emphasize that the proof is exactly the same for non-trivial coefficient \textcolor{Black}{sheaves} as well. See for example \cite{loeffler2021euler},\cite{grossi2020norm}\ \text{\&}\ \cite{Loeffler_2021} for the formalisms required to set things up with non-trivial coefficients. For a lot of the definitions regarding certain constructions appearing below, we refer the reader to \cite[\S $8$ \& \S $9$]{Loeffler_2021}, \cite[\S $5$ \& \S $6$]{grossi2020norm} and \cite[\S $5$]{groutides2024rankinselbergintegralstructureseuler}.

Let $S$ be a finite set of primes containing $2$ and all primes that ramify in $E$. We write $\mathscr{G}^*_S$ for $\mathscr{G}^*(\prod_{p\in S}\mathbf{Q}_p)$. We fix once and for all an open compact subgroup $\mathscr{K}\subseteq\mathscr{G}^*_S$ and an integral element $\delta_S^*\in\mathcal{I}_0(\mathscr{G}^*_S/\mathscr{K},\mathbf{Z})$, where here the $\mathcal{I}$-lattice is defined as usual (e.g. see \cite[\S $3$]{Loeffler_2021} for this more general notation). Let $n$ be a positive integer coprime to $S$, we define the open compact of $\mathscr{G}^*(\mathbf{A}_\mathrm{f})$:
$$\mathscr{K}_S[n]:=\mathscr{K}\times \prod_{\substack{p\notin S\\ p\nmid n}} \mathscr{G}^*(\mathbf{Z}_p) \times\prod_{\substack{p\notin S\\ p|n}}\mathscr{G}^*(\mathbf{Z}_p)[p].$$
For ease of notation, we simply write $\mathscr{H}_S$ for $\mathscr{H}_S[1]$. We follow the notation of \cite[\S $5.2$]{groutides2024rankinselbergintegralstructureseuler} and thus for \textit{any} collection of integral elements $\underline{\delta}^*:=(\delta_p^*)_{p\notin S}$ we set 
$$\delta[n]^*_p:=\begin{dcases}
    \ch(\mathbf{Z}_p^2)\otimes \ch(\mathscr{G}^*(\mathbf{Z}_p)),\ &\mathrm{if}\ p\nmid n\\
    \delta^*_p,\ &\mathrm{if}\ p|n.
\end{dcases}$$
It is important to note that $\delta[n]_p^*$ only depends on primes $p|n$.
Finally (suppressing the dependence on $\delta_S^*$ which is fixed throughout), we write $\delta[n]^*:=\delta_S^*\otimes \bigotimes_{p\notin S}\delta[n]^*_p.$ 
We denote by $Y_{\mathscr{G}^*}$ the infinite level Shimura variety attached to $\mathscr{G}^*$. Similarly to \cite[\S$6.1$]{grossi2020norm}, we have an Asai-Flach map:
$$\mathcal{AF}:\mathcal{S}_0(\mathbf{A}_\mathrm{f}^2,\mathbf{Q})\otimes_\mathbf{Q} C_c^\infty(\mathscr{G}^*(\mathbf{A}_\mathrm{f}),\mathbf{Q})\longrightarrow H^3_\mathrm{mot}(Y_{\mathscr{G}^*}, \mathbf{Q}(2))$$
which is $\mathscr{G}^*(\mathbf{A}_\mathrm{f})\times G(\mathbf{A}_\mathrm{f})$-equivariant in the sense of \cite[\S $3$]{Loeffler_2021}. In particular, it makes sense to evaluate this map on elements in the $\mathcal{I}$-spaces of $G$-coinvariants.
 
\begin{defn}\label{def classes}
    For any collection $\underline{\delta}^*=(\delta_p^*)_{p\notin S}$ of local integral input data with determinant level, $\delta_p^*\in\scalemath{0.9}{\mathcal{I}(\mathscr{G}^*(\mathbf{Q}_p)/\mathscr{G}^*(\mathbf{Z}_p)[p],\mathbf{Z}[1/p])}$, we define the motivic Asai-Flach classes
    \begin{align}\label{af classes}
        z_{\mathrm{mot},n}(\underline{\delta}^*):=\mathcal{AF}\left(\delta[n]^*\right)\in H^3_\mathrm{mot}(Y_{\mathscr{G}^*}(\mathscr{H}_S[n]),\mathbf{Q}(2)),\ \ (n,S)=1.
    \end{align}
\end{defn}
We once again note that for a fixed integral collection $\underline{\delta}^*$, the class $z_{\mathrm{mot},n}(\underline{\delta}^*)$ depends on the local elements $\delta_p^*$ for $p|n$. Working with the smaller group $\mathscr{G}^*$ instead of $\mathscr{G}$, means that we have an embedding $\mathscr{G}^*\hookrightarrow \mathscr{G}^*\times \GL_1, (g,x=\det(g))\mapsto ((g,x),x)$. Thus, for each positive integer $n$ coprime to $S$, we have an induced open and closed embedding $Y_{\mathscr{G}^*}(\mathscr{H}_S[n])\hookrightarrow Y_{\mathscr{G}^*}(\mathscr{H}_S)\times_{\mathrm{Spec}(\mathbf{Q})}\mathrm{Spec}(\mathbf{Q}(\mu_n))$, \textcolor{Black}{defined using the natural map $Y_{\mathscr{G}^*}(\mathscr{H}_S[n])\rightarrow Y_{\mathscr{G}^*}(\mathscr{H}_S)$.} \textcolor{Black}{This} is in fact an isomorphism by \cite[Proposition $5.4.2$]{loeffler2021euler}. Pushing forward the class $z_{\mathrm{mot},n}(\underline{\delta}^*)$ along this map, we obtain classes \begin{align}\label{AF cycl classes}\mathrm{AF}_{\mathrm{mot},n}(\underline{\delta}^*)\in H_{\mathrm{mot}}^3(Y_{\mathscr{G}^*}(\mathscr{H}_S)\times_{\mathrm{Spec}(\mathbf{Q})}\mathrm{Spec}(\mathbf{Q}(\mu_n)),\mathbf{Q}(2)),\ \ (n,S)=1.\end{align}
The action of the Hecke operator  $\ch(\mathscr{G}^*(\mathbf{Z}_p)\ g\ \mathscr{G}^*(\mathbf{Z}_p))$ on the $z$-classes is intertwined with the action of the operator $\ch(\mathscr{G}^*(\mathbf{Z}_p)\ g\ \mathscr{G}^*(\mathbf{Z}_p))\times \mathrm{Frob}_p^{v_p(\det(g))}$ on the $\mathrm{AF}$-classes. Here $\mathrm{Frob}_p$ denotes arithmetic Frobenius at $p$ as an element of $\mathrm{Gal}(\mathbf{Q}(\mu_n)/\mathbf{Q}).$ 
\subsubsection{Integral variants}
Both the $z$-classes and the $\mathrm{AF}$-classes, admit certain integral variants: Let $c\in\mathbf{Z}_{>1}$ coprime to $6S$. Then for any $n\in\mathbf{Z}_{\geq 1}$ coprime to $6cS$, we have classes
\begin{align*}
    _cz_{\mathrm{mot},n}(\underline{\delta}^*)&:=\left(c^2-\prod_{p|c}\mathcal{S}_*(p)^{v_p(c)}\right)\cdot z_{\mathrm{mot},n}(\underline{\delta}^*)\\
    _c\mathrm{AF}_{\mathrm{mot},n}(\underline{\delta}^*)&:=\left(c^2-\prod_{p|c}\mathcal{S}_*(p)^{v_p(c)}\mathrm{Frob}_p^{2v_p(c)}\right)\cdot \mathrm{AF}_{\mathrm{mot},n}(\underline{\delta}^*).
\end{align*}
Of course, these definitions are compatible with one another under pushforward by the intertwining property mentioned above. For more details on how these classes arise, we refer the reader to \cite[\S $7$]{loeffler2021euler}. 

\begin{prop}\label{Prop integral classes}
    Let $c\in\mathbf{Z}_{>1}$ coprime to $6S$ and let $\underline{\delta}^*=(\delta_p^*)_{p\notin S}$, be any collection of local integral input data with determinant level $\delta_p^*\in \mathcal{I}(\mathscr{G}^*(\mathbf{Q}_p)/\mathscr{G}^*(\mathbf{Z}_p)[p],\mathbf{Z}[1/p])$. Then for any prime $\ell$, coprime to $nS$, the image of the class
    $$_c\mathrm{AF}_{\mathrm{mot},n}(\underline{\delta}^*), \ \ \ (n,cS)=1$$
    under the $\ell$-adic \'etale regulator map, is integral. (i.e. lies in the \'etale cohomology with $\mathbf{Z}_\ell$-coefficients).
\end{prop}
\begin{proof}
    The argument is that of \cite[Proposition $9.5.2$]{Loeffler_2021} and also works for non-trivial coefficients.
\end{proof}
    As remarked in \cite{Loeffler_2021}, this global cohomological integrality of the above proposition is what inspired the definition of the local integral lattices, with determinant level. As promised, we can now give even more depth to the integral essence of these classes. We show the local factors appearing in tame norm relations in motivic cohomology, satisfy the expected optimal integral behavior as conjectured by Loeffler. 
    \begin{thm}\label{thm euler system}
        Let $c\in\mathbf{Z}_{>1}$ coprime to $6S$ and let $\underline{\delta}^*=(\delta_p^*)_{p\notin S}$ be any collection of local integral data with determinant level $\delta_p^*\in\mathcal{I}_0(\mathscr{G}^*(\mathbf{Q}_p)/\mathscr{G}^*(\mathbf{Z}_p)[p],\mathbf{Z}[1/p])$. Let $\mathscr{Z}_{cS}$ be the set of square-free positive integers co-prime to $cS$. Then the corresponding collection of motivic \emph{(}integral\emph{)} Asai-Flach classes
        $$_c\mathrm{AF}_{\mathrm{mot},n}(\underline{\delta}^*)\in H_{\mathrm{mot}}^3(Y_{\mathscr{G}^*}(\mathscr{H}_S)\times_\mathbf{Q} \mathbf{Q}(\mu_n),\mathbf{Q}(2)),\ \ \ n\in\mathscr{Z}_{cS}$$
        satisfies the following motivic norm relations:
        \begin{enumerate}
            \item For $n,m\in\mathscr{Z}_{cS}$ with $\tfrac{m}{n}=p$ prime, we have 
            $$\mathrm{norm}^{\mathbf{Q}(\mu_m)}_{\mathbf{Q}(\mu_n)}\left( _c\mathrm{AF}_{\mathrm{mot},m}(\underline{\delta}^*)\right)=\mathcal{P}_{\mathrm{Tr}(\delta_p^*)}^\mathrm{cycl}\cdot\ _c\mathrm{AF}_{\mathrm{mot},n}(\underline{\delta}^*)$$
            with $$\mathcal{P}_{\mathrm{Tr}(\delta_p^*)}^\mathrm{cycl}\in \left\langle p-1,\mathcal{P}_{p,\mathrm{As}_*}^{'}(\mathrm{Frob}_p^{-1})\right\rangle\subseteq \mathcal{H}_{\mathscr{G}^*(\mathbf{Q}_p)}^\circ(\mathbf{Z}[1/p])[\mathrm{Gal}(\mathbf{Q}(\mu_n)/\mathbf{Q})].$$
            \item Let $\nu_p:=p(p-1)^2(p+1)$ and $\phi_{p,2}:=\nu_p\cdot\ch(p^2\mathbf{Z}_p\times (1+p^2\mathbf{Z}_p))\in\mathcal{S}_0(\mathbf{Q}_p^2)$. Finally, set $n_x:=\left[\begin{smallmatrix}
                1 & x \\
                 & 1
            \end{smallmatrix}\right]$. If we specialize the integral collection $\underline{\delta}^*$ to $\underline{\delta}_1^*=(\delta_{1,p}^*)_{p\notin S}$ where 
            $$\delta_{1,p}^*:=\begin{dcases}
                \phi_{p,2}\otimes\left[\ch\left(\mathscr{G}^*(\mathbf{Z}_p)[p]\right)-\ch\left((1,n_{1/p})\ \mathscr{G}^*(\mathbf{Z}_p)[p]\right) \right],\ &\mathrm{if}\ p\ \mathrm{splits}\\
                \phi_{p,2}\otimes\left[\ch\left(\mathscr{G}^*(\mathbf{Z}_p)[p]\right)-\ch\left(n_{\alpha_p/p}\ \mathscr{G}^*(\mathbf{Z}_p)[p]\right) \right],\ &\mathrm{if}\ p\ \mathrm{is}\ \mathrm{inert}
            \end{dcases}$$
            then for all $n,m\in\mathscr{Z}_{cS}$ with $\tfrac{m}{n}=p$ prime, we have $$\mathrm{norm}^{\mathbf{Q}(\mu_m)}_{\mathbf{Q}(\mu_n)}\left( _c\mathrm{AF}_{\mathrm{mot},m}(\underline{\delta}_1^*)\right)=\mathcal{P}_{p,\mathrm{As}_*}^{'}(\mathrm{Frob}_p^{-1})\cdot\ _c\mathrm{AF}_{\mathrm{mot},n}(\underline{\delta}_1^*)$$
        \end{enumerate}
        \begin{proof}
            We firstly prove the first part. Using the same argument as in the proof of \cite[ Theorem $5.2.7$]{groutides2024rankinselbergintegralstructureseuler}, which closely follows the proof of \cite[Theorem $9.4.3$]{Loeffler_2021}, together with the intertwining properties of the pushforward along $Y_{\mathscr{G}^*}(\mathscr{H}_S[n])\hookrightarrow Y_{\mathscr{G}^*}(\mathscr{H}_S)\times_\mathbf{Q}\mathbf{Q}(\mu_n)$, we can reduce this to the following local problem: We need to show that for every local component $\delta_p^*$ as in the statement of the theorem, the corresponding Hecke operator $\mathcal{P}_{\mathrm{Tr}(\delta_p^*)}$ (of \Cref{def freeness*}) lies in the ideal of $\mathcal{H}_{\mathscr{G}^*(\mathbf{Q}_p)}^\circ(\mathbf{Z}[1/p])$ given by $(p-1,\mathcal{P}_{p,\mathrm{As}_*}^{'}(1))$. 
            
            \noindent We have the following commutative diagram: 
     \[\begin{tikzcd}[ampersand replacement=\&,cramped,sep=scriptsize]
	{\mathcal{I}_0(\mathscr{G}^*(\mathbf{Q}_p)/\mathscr{G}^*(\mathbf{Z}_p)[p],\mathbf{Z}[1/p])} \& {\mathcal{I}_0(\mathscr{G}^*(\mathbf{Q}_p)/\mathscr{G}^*(\mathbf{Z}_p),\mathbf{Z}[1/p])} \& {\mathcal{I}(\mathscr{G}^*(\mathbf{Q}_p)/\mathscr{G}^*(\mathbf{Z}_p))} \& {\mathcal{H}_{\mathscr{G}^*(\mathbf{Q}_p)}^\circ} \\
	{\mathcal{I}_0(\mathscr{G}(\mathbf{Q}_p)/\mathscr{G}(\mathbf{Z}_p)[p],\mathbf{Z}[1/p])} \& {\mathcal{I}_0(\mathscr{G}(\mathbf{Q}_p)/\mathscr{G}(\mathbf{Z}_p),\mathbf{Z}[1/p])} \& {\mathcal{I}(\mathscr{G}(\mathbf{Q}_p)/\mathscr{G}(\mathbf{Z}_p))} \& {\mathcal{H}_{\mathscr{G}(\mathbf{Q}_p)}^\circ}
	\arrow["{\mathrm{Tr}}", from=1-1, to=1-2]
	\arrow["{i_p[p]}", from=1-1, to=2-1]
	\arrow[hook', from=1-2, to=1-3]
	\arrow["{i_p}", hook, from=1-2, to=2-2]
	\arrow["\simeq", from=1-3, to=1-4]
	\arrow["{i_p}", hook, from=1-3, to=2-3]
	\arrow["{\iota_p}", hook, from=1-4, to=2-4]
	\arrow["{\mathrm{Tr}}", from=2-1, to=2-2]
	\arrow[hook, from=2-2, to=2-3]
	\arrow["\simeq", from=2-3, to=2-4]
\end{tikzcd}\]
where the trace maps $\mathrm{Tr}$ are given by $\phi\otimes \ch(g\mathscr{G}^*(\mathbf{Z}_p)[p])\mapsto \phi\otimes \ch(g\mathscr{G}^*(\mathbf{Z}_p))$, and similarly for $\mathscr{G}$ instead of $\mathscr{G}^*$. The vertical maps $i_p$ , are given by $\phi\otimes \ch(g\mathscr{G}^*(\mathbf{Z}_p))\mapsto \phi\otimes \ch(g\mathscr{G}(\mathbf{Z}_p))$ and similarly for $i_p[p]$. Notice how this gives well-defined maps at both integral levels since the subgroup $\mathrm{Stab}_{\GL_2(\mathbf{Q}_p)}(\phi)\cap g \mathscr{G}^*(\mathbf{Z}_p)g^{-1}$ coincides with $\mathrm{Stab}_{\GL_2(\mathbf{Q}_p)}(\phi)\cap g \mathscr{G}(\mathbf{Z}_p)g^{-1}$ and similarly for the determinant levels. Let $\delta_p:=\left(i_p[p]\right)(\delta_p^*)$. It follows from the diagram that 
$$\iota_p\left(\mathcal{P}_{\mathrm{Tr}(\delta_p^*)}\right)=\mathcal{P}_{\mathrm{Tr}(\delta_p)}.$$
However, since the element $\delta_p$ is contained in $\mathcal{I}_0(\mathscr{G}(\mathbf{Q}_p)/\mathscr{G}(\mathbf{Z}_p)[p],\mathbf{Z}[1/p])$, we have already studied the integral behavior of the Hecke operator $\mathcal{P}_{\mathrm{Tr}(\delta_p)}$ in detail. Indeed, by \cite[Theorem $5.1.1$]{groutides2024rankinselbergintegralstructureseuler} if $p$ splits, and by \Cref{thm integral local factors} if $p$ is inert, we see that $\mathcal{P}_{\mathrm{Tr}(\delta_p)}$ and hence $\iota_p\left(\mathcal{P}_{\mathrm{Tr}(\delta_p^*)}\right)$ is contained in the ideal of $\mathcal{H}_{\mathscr{G}(\mathbf{Q}_p)}^\circ(\mathbf{Z}[1/p])$ generated by $p-1$ and $\iota_p(\mathcal{P}_{p,\mathrm{As}_*}^{'}(1))$. In the inert case, this is enough to deduce that $\mathcal{P}_{\mathrm{Tr}(\delta_p^*)}$ is indeed contained in $(p-1,\mathcal{P}_{p,\mathrm{As}_*}^{'}(1))\subseteq\mathcal{H}_{\mathscr{G}^*(\mathbf{Q}_p)}^\circ(\mathbf{Z}[1/p])$ since in this case, $\iota_p$ is also surjective. \textcolor{Black}{This is because for inert $p$, the groups $\mathscr{G}(\mathbf{Q}_p)$ and $\mathscr{G}^*(\mathbf{Q}_p)$ have exactly the same coset representatives $t(n_1,n_2), n_1\geq n_2$, in their respective Cartan decompositions.} 

In the split case, things are once again a bit more delicate. For a split prime $p$, the element $\delta_p^*$ is an integral linear combination of integral elements of the form $\phi\otimes \ch(g\mathscr{G}^*(\mathbf{Z}_p)[p])$ with $g\in\mathscr{G}^*(\mathbf{Q}_p)=G(\mathbf{Q}_p)\times_{\mathrm{Q}_p^\times}G(\mathbf{Q}_p)$. Thus, $\delta_p$ is an integral linear combination of integral elements $\phi\otimes \ch(g\mathscr{G}(\mathbf{Z}_p)[p])$ where $g$ is still an element of $\mathscr{G}^*(\mathbf{Q}_p)$. In the notation of \cite[\S $4.3$]{groutides2024rankinselbergintegralstructureseuler}, the image of $\mathrm{Tr}(\delta_p)$, under $\Xi_c$, in $C_c^\infty(P(\mathbf{Q}_p)\backslash\mathscr{G}(\mathbf{Q}_p)/\mathscr{G}(\mathbf{Z}_p))$ is given by an integral linear combination of characteristic functions $f_{z_0,\gamma_0}:=\ch(P(\mathbf{Q}_p)z_0(\gamma_0,1)\mathscr{G}(\mathbf{Z}_p))$ where $$z_0=\left[\begin{smallmatrix}
    p^{r_0} & \\
     & p^{r_0}
\end{smallmatrix}\right]\in Z(\mathbf{Q}_p), \gamma_0=\left[\begin{smallmatrix}
    p^{r_1} & \\
     & p^{r_1}
\end{smallmatrix}\right]\left[\begin{smallmatrix}
    p^m & y \\
    & 1
\end{smallmatrix}\right]\in \GL_2(\mathbf{Q}_p).$$
However, since now $g\in\mathscr{G}^*(\mathbf{Q}_p)$, we will have $2r_1+m=0$. In addition, $\mathcal{P}_{\mathrm{Tr}(\delta_p)}^{'}$ will be an integral linear combination of the Hecke operators $\mathscr{V}_{z_0,\gamma_0}[\mathscr{P}_{\gamma_0}+\mathscr{D}_{\gamma_0}\mathcal{P}_p(1)]$ appearing in \cite[Proposition $4.3.6$]{groutides2024rankinselbergintegralstructureseuler} (recall that for split $p$, we have $\iota_p(\mathcal{P}_{p,\mathrm{As}_*}(1))=\mathcal{P}_p(1)$). Using the fact that $2r_1+m=0$ and the explicit expressions in \textit{loc.cit}, a simple counting argument shows that $\mathscr{V}_{z_0,\gamma_0}\mathscr{P}_{\gamma_0}$ and $\mathscr{V}_{z_0,\gamma_0}\mathscr{D}_{\gamma_0}$ both lie in the image of $\iota_p$. Thus, we conclude the proof of the first part of the theorem.

\noindent For the second part of the theorem, we let $\delta_{1,p}:=(i_p[p])(\delta_{1,p}^*)$. Thus
$$\mathrm{Tr}(\delta_{1,p})=\begin{dcases}
                \phi_{p,2}\otimes\left[\ch\left(\mathscr{G}(\mathbf{Z}_p)[p]\right)-\ch\left((1,n_{1/p})\ \mathscr{G}(\mathbf{Z}_p)[p]\right) \right],\ &\mathrm{if}\ p\ \mathrm{splits}\\
                \phi_{p,2}\otimes\left[\ch\left(\mathscr{G}(\mathbf{Z}_p)[p]\right)-\ch\left(n_{\alpha_p/p}\ \mathscr{G}(\mathbf{Z}_p)[p]\right) \right],\ &\mathrm{if}\ p\ \mathrm{is}\ \mathrm{inert}\end{dcases}
    $$
    in the case where $\tfrac{m}{n}=p$ is split, this is covered in \cite[Theorem $5.2.7(3)$]{groutides2024rankinselbergintegralstructureseuler} and in this case, we have $\mathcal{P}_{\mathrm{Tr}(\delta_{1,p})}=\mathcal{P}_p^{'}(1)=\iota_p(\mathcal{P}_{p,\mathrm{As}_*}^{'}(1))$ as required. Now assume that $\tfrac{m}{n}=p$ is inert. We once again want to realize $\mathcal{P}_{\mathrm{Tr}(\delta_{1,p})}$ as $\iota_p(\mathcal{P}_{p,\mathrm{As}_*}^{'}(1))=\mathcal{P}_{E_p}^{'}(1)$. Using the approach in the proof of \cref{prop integral period 1} (i.e. essentially \cite[Proposition $4.4.1$]{Loeffler_2021}), it is convenient to do this using the Asai period $\mathcal{Z}$ of \Cref{def periods}. Let $F:=E_p$ and $\alpha:=\alpha_p$. By the usual density argument, it suffices to show that for any unramified principal series $\Pi_{F}=\mathcal{W}(\Pi_F,\psi_F)$ of $G(F)=\mathscr{G}(\mathbf{Q}_p)$, we have
    \begin{align}\label{eq: 28}A(s):=Z\left(\phi_{p,2},W_{\Pi_{F}}^\mathrm{sph}-n_{\alpha/p}W_{\Pi_{F}}^\mathrm{sph},s\right)=1.\end{align}
    where $Z$ is the local Asai zeta integral of \Cref{def Asai zeta integral}. We give some of the details here. As in \cite[\S $4$]{flicker1988twisted}, we can unfold the integral in \eqref{eq: 28} as follows:
    \begin{align}
        \nonumber A(s)=\int_{G(\mathbf{Z}_p)}&\left\{\int_{\mathbf{Q}_p^\times}\left[W_{\Pi_F}^\mathrm{sph}(\left[\begin{smallmatrix}
            y & \\
            & 1
        \end{smallmatrix}\right])-W_{\Pi_F}^\mathrm{sph}(\left[\begin{smallmatrix}
            y & \\
            & 1
        \end{smallmatrix}\right]k\ n_{\alpha/p})\right]|y|_p^{s-1}\ d^\times y\right\}  \\
       \label{eq:30}&\times \left\{\int_{\mathbf{Q}_p^\times}\omega_{\Pi_F}(x)\phi_{p,2}((0,x)k)|x|_p^{2s}\ d^\times x\right\}\ dk.
    \end{align}
    One cheks that the Godement-Siegel section appearing in \eqref{eq:30}, regarded as a function on $G(\mathbf{Z}_p)$, is supported on $K_0(p^2):=\{g\in G(\mathbf{Z}_p)\ |\ g\in\left[\begin{smallmatrix}
        * & * \\
        p^2\mathbf{Z}_p & *
    \end{smallmatrix}\right]\}$ where it takes the constant value $\tfrac{\nu_p}{p(p-1)}$. Thus
    \begin{align*}
        A(s)&=\tfrac{\nu_p}{p(p-1)}\int_{K_0(p^2)}\int_{\mathbf{Q}_p^\times}\left[W_{\Pi_F}^\mathrm{sph}(\left[\begin{smallmatrix}
            y & \\
            & 1
        \end{smallmatrix}\right])-W_{\Pi_F}^\mathrm{sph}(\left[\begin{smallmatrix}
            y & \\
            & 1
        \end{smallmatrix}\right]k\ n_{\alpha/p})\right]|y|_p^{s-1}\ d^\times y.
    \end{align*}
    We set $$K_{0,1}^1(p^2):=\left\{g\in G(\mathbf{Z}_p)\ |\ g\in\left[\begin{smallmatrix}
        1+p^2\mathbf{Z}_p & * \\
        p^2\mathbf{Z}_p & 1+p^2\mathbf{Z}_p
    \end{smallmatrix}\right]\right\}\subseteq K_0(p^2).$$
    We have that $n_{\alpha/p}^{-1}\cdot K_{0,1}^1(p^2)\cdot  n_{\alpha/p}\subseteq G(\mathcal{O}_F)$ and $\{\left[\begin{smallmatrix}
        i & \\
        & j
    \end{smallmatrix}\right]\ |\ i,j\in(\mathbf{Z}/p^2\mathbf{Z})^\times\}$ is a complete set of distinct coset representatives for $K_0(p^2)/K_{0,1}^1(p^2)$. Hence after a bit of rearranging, we obtain
    \begin{align*}
       A(s)&=\nu_p\frac{\vol_{G(\mathbf{Q}_p)}(K_{0,1}^1(p^2))}{p(p-1)}\sum_{i,j\in(\mathbf{Z}/p^2\mathbf{Z})^\times}\int_{\mathbf{Q}_p^\times}\left[W_{\Pi_F}^\mathrm{sph}(\left[\begin{smallmatrix}
            y & \\
            & 1
        \end{smallmatrix}\right])-W_{\Pi_F}^\mathrm{sph}\left(\left[\begin{smallmatrix}
            y & \\
            & 1
        \end{smallmatrix}\right]\left[\begin{smallmatrix}
            1 & \tfrac{i\alpha}{jp}\\
            & 1
        \end{smallmatrix}\right]\right)\right]|y|_p^{s-1}\ d^\times y\\
        &=\nu_p\frac{\vol_{G(\mathbf{Q}_p)}(K_{0,1}^1(p^2))}{p(p-1)}\sum_{i,j\in(\mathbf{Z}/p^2\mathbf{Z})^\times}\int_{\mathbf{Z}_p^\times} 1-\psi_F(\alpha\ i j ^{-1} p^{-1} y)\ d^\times y\\
        &=\nu_p\frac{\vol_{G(\mathbf{Q}_p)}(K_{0,1}^1(p^2))}{p(p-1)}\sum_{i,j\in(\mathbf{Z}/p^2\mathbf{Z})^\times}\tfrac{p}{p-1}\\
        &=\nu_pp^2\vol_{G(\mathbf{Q}_p)}(K_{0,1}^1(p^2)).
    \end{align*}
    where the second and third equality essentially follows from the fact that $\psi_F(\alpha(-))$, as a character of $\mathbf{Q}_p$, has conductor $\mathbf{Z}_p$, and the spherical Whittaker functions vanish on $\left[\begin{smallmatrix}
        y & \\
        & 1
    \end{smallmatrix}\right]$ for $y\notin \mathbf{Z}_p$. Finally, one can directly compute the volume factor, which turns to be precisely equal to $\tfrac{1}{p^2\nu_p}$ giving us the result. It is also quite straightforward to see that $\vol_{G(\mathbf{Q}_p)}(K_{0,1}^1(p^2))$, coincides with $\tfrac{1}{p}\cdot\vol_{G(\mathbf{Q}_p)}(\mathrm{Stab}_{G(\mathbf{Q}_p)}(\phi_{p,2})\cap n_{\alpha/p}\mathscr{G}^*(\mathbf{Z}_p)[p]n_{\alpha/p}^{-1})$ if $p$ is inert, and $\tfrac{1}{p}\cdot\vol_{G(\mathbf{Q}_p)}(\mathrm{Stab}_{G(\mathbf{Q}_p)}(\phi_{p,2})\cap (1,n_{1/p})\mathscr{G}^*(\mathbf{Z}_p)[p](1,n_{1/p})^{-1})$ if $p$ is split. Thus, the volume calculation also shows that indeed, $\delta_{p,1}^*\in\mathcal{I}_0(\mathscr{G}^*(\mathbf{Q}_p)/\mathscr{G}^*(\mathbf{Z}_p)[p],\mathbf{Z}[1/p])$ for every $p\notin S$, as claimed. 
        \end{proof}
    \end{thm}
\begin{rem}
     One of the things evident from the integral computations above in the nice case of $\delta_{1,p}$, is that it's futile to try and attack the generality of the first part of \Cref{thm euler system}, or \Cref{prop integral period 1}, by directly working with the local Asai period $\mathcal{Z}$. Although at first glance it seems true, one can not hope to directly obtain any meaningful results at an integral level for arbitrary integral input data $\phi\otimes gW_{\Pi_F}^\mathrm{sph}$, by direct computations and unfolding. That is why the alternative Hecke-theoretic strategy to overcome this, which was introduced in \cite{groutides2024rankinselbergintegralstructureseuler} and further adapted here, is necessary. 
\end{rem}

 \bibliography{citation} 
\bibliographystyle{alpha}

\noindent\textit{Mathematics Institute, Zeeman Building, University of Warwick, Coventry CV4 7AL,
England}.\\
\textit{Email address}: Alexandros.Groutides@warwick.ac.uk
\end{document}